\documentclass[letterpaper]{siamltex}
\usepackage{fullpage,enumerate,setspace,changepage,multirow,rotating,color}
\usepackage{graphicx,amsmath,amssymb,url,algorithm,algorithmic,array}
\def\grad{\nabla}

\def\bo{\mathbf{0}}

\def\cK{\mathcal{K}}

\def\cO{\mathcal{O}}

\def\cX{\mathcal{X}}
\def\cY{\mathcal{Y}}

\def\mE{\mathbb{E}}

\def\eq{\[}
\def\en{\]}

\def\smskip{\smallskip}

\def\texitem#1{\par\smskip\noindent\hangindent 25pt
               \hbox to 25pt {\hss #1 ~}\ignorespaces}

\def\abs#1{\left|#1\right|}
\def\norm#1{\|#1\|}

\newcommand{\BEAS}{\begin{eqnarray*}}
\newcommand{\EEAS}{\end{eqnarray*}}
\newcommand{\BEA}{\begin{eqnarray}}
\newcommand{\EEA}{\end{eqnarray}}
\newcommand{\BEQ}{\begin{eqnarray}}
\newcommand{\EEQ}{\end{eqnarray}}
\newcommand{\BIT}{\begin{itemize}}
\newcommand{\EIT}{\end{itemize}}
\newcommand{\BNUM}{\begin{enumerate}}
\newcommand{\ENUM}{\end{enumerate}}

\newcommand{\BA}{\begin{array}}
\newcommand{\EA}{\end{array}}

\newcommand{\ones}{\mathbf 1}

\newcommand{\reals}{\mathbb{R}}
\newcommand{\integers}{\mathbb{Z}}

\newcommand{\Rank}{\mathop{\bf rank}}

\newcommand{\argmin}{\mathop{\rm argmin}}
\newcommand{\argmax}{\mathop{\rm argmax}}

\newif\ifpagenumbering
\pagenumberingtrue

\pagenumberingfalse

\newsavebox{\theorembox}
\newsavebox{\lemmabox}
\newsavebox{\remarkbox}
\savebox{\theorembox}{\noindent\bf Theorem}
\savebox{\lemmabox}{\noindent\bf Lemma}
\savebox{\remarkbox}{\noindent\bf Remark}

\def\fprod#1{\langle#1 \rangle}
\def\proj#1{\pi_\Omega\left(#1\right)}
\def\projc#1{\pi_{\Omega^c}\left(#1\right)}
\newcommand{\sgn}{\mathrm{sgn}}
\newcommand{\Diag}{\mathop{\bf Diag}}
\newcommand\redsout{\bgroup\markoverwith{\textcolor{red}{\rule[1ex]{2pt}{0.8pt}}}\ULon}
\newcommand{\admip}{\texttt{ADMIP}}
\newcommand{\admm}{\texttt{ADMM}}
\newcommand{\asalm}{\texttt{ASALM}}

\makeatletter
\newcommand{\thickhline}{%
    \noalign {\ifnum 0=`}\fi \hrule height 1pt
    \futurelet \reserved@a \@xhline
}
\makeatother

{\textsc{Lemma} \ref{#1} \begin{itshape}}%
{\end{itshape}}
{\textsc{Theorem} \ref{#1} \begin{itshape}}%2
{\end{itshape}}

\title{An Alternating Direction Method with Increasing Penalty for Stable Principal Component Pursuit}%\footnotemark[4]}
\author{
    N. S. Aybat \footnotemark[2]
    \and
    G. Iyengar \footnotemark[3]\ %\footnotemark[4]
}
\begin{document}
\maketitle
\renewcommand{\thefootnote}{\fnsymbol{footnote}}
\footnotetext[2]{IE Department, The Pennsylvania State University. Email: {\tt nsa10@psu.edu}. Supported by NSF grant CMMI-1400217.}
\footnotetext[3]{IEOR Department, Columbia University. Email: {\tt gi10@columbia.edu}. Supported by NIH R21 AA021909-01, NSF CMMI-1235023, NSF DMS-1016571 grants.}
%\footnotetext[4]{Research partially supported by ONR grant N000140310514, NSF Grant DMS 10-16571 and DOE Grant DE-FG02-08-25856.}
\renewcommand{\thefootnote}{\arabic{footnote}}
\begin{abstract}
 The stable principal component pursuit~(SPCP) is a non-smooth
 convex optimization problem, the solution of which
 enables one to reliably recover the low rank and sparse components of a
 data matrix which is corrupted by a dense noise matrix, % The recovery
   % properties of SPCP extends to the case where
   even when
   only a fraction of data
   entries are observable. %, both in theory and
 % in practice.
 In this paper, we propose %a new first order non-smooth adaptive
 % augmented Lagrangian algorithm NSA to solve
 % the SPCP problem.
 a new algorithm for solving SPCP. The proposed algorithm is a modification of the
   alternating direction method of multipliers~(\admm) where we use an increasing sequence of
   penalty parameters instead of a fixed  penalty.
% to solve the SPCP
%    problem. Proposed ADMM
   The algorithm is based on partial variable
   splitting and works directly with the non-smooth objective function. We
   show that both primal and dual iterate sequences converge under mild
   conditions on the % increasing
   sequence of penalty parameters.
 % NSA is an alternating direction method, with partial variable splitting,
 % customized for the SPCP problem.  \ADMIP~ works directly with the
 % non-smooth objective function and converges under mild conditions on the
 % sequence of penalty multipliers, in particular, the penalty multipliers
 % are allowed to grow unbounded.
 To the best of our knowledge, % that this is the
   % first result
   % of its kind in the literature
   this is the first convergence result for a
   variable penalty \admm~when penalties are not bounded, the
   objective function is non-smooth and its sub-differential is not
   uniformly bounded. Using partial variable
splitting and adopting an increasing sequence of penalty multipliers,
together,   %grow unbounded
significantly reduce the number of iterations required to achieve feasibility in practice.
Our preliminary computational tests show that the proposed algorithm
  works very well in practice, and
  outperforms \asalm, a state of the art \admm~algorithm for the SPCP
  problem with a constant penalty parameter.  %To best of our knowledge, an algorithm for the SPCP problem that has $\cO(1/\epsilon)$ iteration complexity and has a per iteration complexity equal to that of a singular value decomposition is given for the first time.
\end{abstract}
%\begin{acknowledgement}
%We would like to thank to Min Tao for providing the code \proc{ASALM}.
%\end{acknowledgement}
\section{Introduction}
Suppose a matrix $D\in\reals^{m\times n}$ is of the form  $D = L^0 + S^0$,
where $L^0$ is a low-rank matrix, i.e. $\Rank(L^0)\ll\min\{m,n\}$, and
$S^0$ is a sparse matrix. The matrix $S^0$ is interpreted as
gross errors in the measurement of the low rank matrix $L^0$. Wright et al.~\cite{Wright09_1J},
 Cand{\'e}s et al.~\cite{Can09_1J} and Chandrasekaran et al.~\cite{Chandrasekaran-Sanghavi-Parrilo-Willsky-2009} proposed recovering the
low-rank $L^0$ and sparse $S^0$ by solving the \emph{principal
  component pursuit}~(PCP) problem
\vspace{-3mm}
\begin{align}
\min_{L\in\reals^{m\times n}} \norm{L}_* + \xi~\norm{D-L}_1, \label{eq:component_pursuit}
\end{align}
where $\xi=\frac{1}{\sqrt{\max\{m,n\}}}$. Here the nuclear norm $\norm{L}_* := \sum_{i = 1}^{r} \sigma_i(L)$, where $\{\sigma_i(L)\}_{i=1}^{r}$ denotes the singular values of $L \in \reals^{m\times n}$, and the $\ell_1$-norm
  $\norm{L}_1:=\sum_{i=1}^m\sum_{j=1}^n|L_{ij}|$.
%  $\norm{L}_\infty:=\max\{|L_{ij}|:~1\leq i\leq m,~ 1\leq j\leq n\}$ and $\norm{L}_2:=\sigma_{\rm max}(L)$, where $\sigma_{\rm max}(L)$ denotes the maximum singular value of $L$.
% %$\norm{L}_1$ is equal to the sum of absolute values of the elements of $L$
% To be more precise about the recovery, let $L^0\in\reals^{m\times n}$ with $\Rank(L^0)=r$ and let Suppose that for some $\mu>0$, $U$ and $V$ satisfy
% \begin{align}
% \label{eq:assumption}
% \max_i\norm{U^Te_i}_2^2\leq \frac{\mu r}{m}, \quad \max_i\norm{V^Te_i}_2^2\leq \frac{\mu r}{n}, \quad \norm{UV^T}_\infty\leq \sqrt{\frac{\mu r}{mn}},
% \end{align}
% where $e_i$ denotes the $i$-th unit vector.
\begin{theorem}~\cite{Can09_1J}
Suppose $D=L^0+S^0 \in \reals^{m \times n}$. Let $r = \Rank(L^0)$ and
$L^0=U\Sigma V^T=\sum_{i=1}^r\sigma_iu_iv_i^T$ denote the singular
value decomposition~(SVD) of $L^0$. Suppose there exists $\mu>0$ such
that
\begin{align}
\label{eq:assumption}
\max_i\norm{U^Te_i}_2^2\leq \frac{\mu r}{m}, \quad
\max_i\norm{V^Te_i}_2^2\leq \frac{\mu r}{n}, \quad
\norm{UV^T}_\infty\leq \sqrt{\frac{\mu r}{mn}},
\end{align}
where $e_i$ denotes the $i$-th unit vector, and
 % where $L^0\in\reals^{m\times n}$ with $m<n$ satisfies \eqref{eq:assumption} for some $\mu>0$,
the non-zero components of the sparse matrix $S^0$ are chosen
uniformly at random. Then there exist constants $c$, $\rho_r$,  and $\rho_s$,
such that the solution of the  PCP
problem~\eqref{eq:component_pursuit} exactly recovers $L^0$ and $S^0$
with probability of at least $1-c n^{-10}$, provided
\begin{align}
\label{eq:assumption2}
\Rank(L^0)\leq\rho_r m \mu^{-1} (\log(n))^{-2} \quad \mbox{and} \quad
\norm{S^0}_0 \leq \rho_s mn,
\end{align}
where the $\ell_0$-norm $\norm{S^0}_0$ denotes the number of non-zero components
of the matrix $S^0$.
\end{theorem}

Now, suppose the data matrix $D$ is of the form $D = L^0 + S^0 + N^0$ such that $L^0$ is a low-rank matrix, $S^0$ is a sparse gross ``error'' matrix, $N^0$ is a dense noise matrix with
$\norm{N^0}_F\leq\delta$, where the Frobenius norm $\norm{Z}_F := \sqrt{
\sum_{i=1}^m\sum_{j=1}^n Z_{ij}^2}$.
In \cite{Can10_1J}, it was shown that it was still possible to recover
the low-rank and sparse components $(L^0,S^0)$ of $D$ by solving
the \emph{stable principal component pursuit}~(SPCP) problem
%Suppose that the data matrix $D$ is of the form $D = L^0 + S^0 +
%N_0$, where where $L^0$ is a low-rank matrix, $S^0$ is a sparse
%matrix as given in the \emph{principal component pursuit} problem and
%$N_0$ is such that $\norm{N_0}_F\leq\delta$. Then solving the
%\emph{stable principal component pursuit} problem
\begin{align}
 \min_{L,S\in\reals^{m\times n}}\{\norm{L}_*+\xi~\norm{S}_1:\
 \norm{L+S-D}_F\leq\delta\}. \label{eq:stable_component_pursuit}
\end{align}
\begin{theorem}~\cite{Can10_1J}
\label{thm:candes2}
Suppose $D = L^0 + S^0 + N^0$, where $L^0\in\reals^{m\times n}$
with $m<n$ satisfies \eqref{eq:assumption} for some $\mu>0$, and the
non-zero components of the sparse matrix $S^0$ are chosen
uniformly at random. Suppose $L^0$ and $S^0$ satisfy
\eqref{eq:assumption2}. Then for any $N^0$ such that
$\norm{N^0}_F\leq\delta$, the solution $(L^*,S^*)$ to the % stable
% principal component pursuit
SPCP
problem~\eqref{eq:stable_component_pursuit} satisfies
$\norm{L^*-L^0}_F^2+\norm{S^*-S^0}_F^2\leq Cmn\delta^2$ for some
constant $C$ with high probability.
\end{theorem}

In many applications, some of the entries of $D$ in
\eqref{eq:stable_component_pursuit} may not be available. Let
$\Omega\subset\{i:1\leq i\leq m\}\times\{j:1\leq j\leq n\}$ be the
index set of the observable entries of $D$. Define the
projection operator $\pi_{\Omega}:\reals^{m\times n}\rightarrow
\reals^{m\times n}$ as follows
\vspace{-5mm}
\begin{equation}
  \label{eq:pi-def}
  (\pi_\Omega(L))_{ij} =  \left\{
    \begin{array}{ll}
      L_{ij} , & (i,j)\in\Omega,\\
      0, & \mbox{otherwise}.
    \end{array}
  \right.
\end{equation}
\vspace{-4mm}

\noindent Note that the adjoint operator  $\pi^*_{\Omega} =\pi_{\Omega}$.
For applications with missing observations, Tao and Yuan~\cite{Tao09_1J} proposed recovering the low rank and
sparse components of $D$ by
solving
\vspace{-2mm}
\begin{equation}
\label{prob:SPCP-missing}
\quad \min_{L,S\in\reals^{m\times n}}\{\|L\|_*+\xi\|S\|_1:\ \|\pi_\Omega(L+S-D)\|_F\leq\delta\}.
\end{equation}
\vspace{-5mm}

\noindent PCP and SPCP both have numerous applications in diverse fields such as video surveillance and face recognition in image processing~\cite{Can09_1J}, and clustering in machine learning~\cite{Aybat15_1P} to name a few. \eqref{eq:component_pursuit},
\eqref{eq:stable_component_pursuit} and \eqref{prob:SPCP-missing} can
be reformulated as semidefinite programming~(SDP) problems, and
therefore, in theory they can be solved in polynomial time
using interior point algorithms; however, these algorithms require very large amount of memory, and  are, therefore, impractical for solving large instances.
%    due to their memory intensive computational requirements}.
Recently, a number of first-order algorithms have been proposed to
solve PCP and SPCP. For existing approaches
to solve PCP and SPCP problems see
~\cite{Aybat13_1j,Ser10_1J,Can09_1J,Gold10_1J,Ma09_1J,Ma09_1R,Tao09_1J,Wright09_1J,Can10_1J} and
references therein.
\subsection*{Our contribution}
  We propose a new \emph{alternating direction method of
  multipliers}~(\admm) with an \emph{increasing penalty}
  sequence called \admip\footnote{In an earlier preprint, we named it as \texttt{NSA} algorithm.}~to solve the SPCP
  problem~\eqref{prob:SPCP-missing}.
  % -\admip~stands for \emph{alternating direction method} with
  % \emph{increasing penalty}
  The \admip\ algorithm, detailed in Figure~\ref{alg:nsa},
  uses \emph{partial} variable splitting on \eqref{prob:SPCP-missing}, and works
  directly with the \emph{non-smooth} objective function. In the context of
  \emph{method of multipliers}, where the primal iterates are computed by minimizing the augmented Lagrangian function, under assumptions related to strong second-order conditions for optimality, it was
  shown in~\cite{Rockafellar1976,Rockafellar-76} that the primal and dual
  iterates  converge to an optimal pair \emph{superlinearly}
  when the penalty parameters $\rho_k\nearrow\infty$, while the rate is
  only \emph{linear} when $\sup_k \rho_k<\infty$. However, this result
  has not been extended to % the context of
  \admm. In a recent survey, % paper on \admm,}
  Boyd et al.~\cite{Boyd11_1J} (see  Section 3.4.1)
  % the penalty multipliers in ADMM are held constant and
  % the authors
  remark % in Section 3.4.1
  that it is difficult to prove the convergence of
  \admm~when penalty multipliers change in every iteration.
  We show that both
  primal and dual \admip\ iterates converge to an
  optimal primal-dual solution for \eqref{prob:SPCP-missing} under mild
  conditions on the penalty multiplier sequence. To the best of our knowledge, % that this is the
    % first result
    % of its kind in the literature
    this is the first convergence result for a
    variable penalty \admm~when penalties are \emph{not bounded}, the
    objective function is \emph{non-smooth} and its subdifferential is
    \emph{not}
    uniformly bounded.

    The work of He et al.~\cite{He98_1J,He00_1J,He02_1J} on variable
      penalty \admm~algorithms implicitly assumes that
  % Even
  % in the case where the multipliers are allowed to change in every
  % iteration, either
  both terms in the
  objective function are % implicitly assumed to be
  \emph{differentiable}; %  -indeed, the ADMM algorithm proposed
  % in~\cite{He98_1J} is for linearly constrained variational inequalities
  % with a single-valued continuous self-map
  % $F:\reals^n\rightarrow\reals^n$;
  therefore, these results do not extend to
  \emph{non-smooth} optimization problem in
  \eqref{eq:problem_generic_split}, i.e. to the \admm~formulation of
  \eqref{prob:SPCP-missing}. The
     variable penalty \admm~algorithms in~\cite{He98_1J,He00_1J,He02_1J}
    % are proposed for solving
     are proposed to solve variational inequalities~(VI) of the form:\vspace{-2mm}
  \begin{equation*}
  (x-x^*)^\top F(x^*)+(y-y^*)^\top G(y^*) \geq 0,\quad \forall
  (x,y)\in\Omega:=\{(x,y):~x\in\cX,~y\in\cY,~Ax+By=b\},%\subset\reals^{n_1+n_2},
  \end{equation*}
  \vspace{-5mm}
  
  \noindent where $A\in\reals^{m\times n_1}$, $B\in\reals^{m\times n_2}$, and
  $b\in\reals^m$. The convergence proofs in~\cite{He98_1J,He00_1J,He02_1J}
  require that both $F:\cX\rightarrow\reals^{n_1}$ and
  $G:\cY\rightarrow\reals^{n_2}$ are \emph{continuous point-to-point maps}
  that are monotone with respect to the non-empty closed convex sets
  $\cX\subset\reals^{n_1}$ and $\cY\subset\reals^{n_2}$, respectively. % In
  % particular, when
  When these variable penalty \admm~methods for VI are applied to the VI
  reformulation of
  convex optimization problems of the form $\min\{f(x)+g(y):\
  (x,y)\in\Omega\}$, the requirement that $F$ and $G$ be continuous
  point-to-point maps
  % requirement in
  % the convergence proofs of~\cite{He98_1J,He00_1J,He02_1J} is equivalent
  % to requiring
  implies that $F(x)=\grad f(x)$, and $G(y)=\grad g(y)$. % However, in the
  % problem we consider both terms in the objective function
  On the other hand, if $f(x)$ and $g(x)$ are non-smooth
  convex functions, % . Therefore, if we write the problem as a variational
  % inequality, we need to use point-to-set maps, i.e.
  then both $F$ and $G$ should be \emph{point-to-set maps}, i.e.,
  \emph{multi-functions}; %. Hence,
  therefore, the convergence proofs for variable penalty \admm~ algorithms
  % \emph{variable penalty}
  in~\cite{He98_1J,He00_1J,He02_1J} do not % hold
  % for
  extend to our problem which is a non-smooth convex optimization problem
  -- see Assumption~A and the following discussion on page 107 in
  \cite{He02_1J}. The \admm~algorithm  in~\cite{Kont98_1J} can solve
    $\min\{f(x)+g(y):\
  (x,y)\in\Omega\}$ when both $f$ and $g$ are non-smooth convex functions;
  however, the convergence proof requires that
  the penalty sequence $\{\rho_k\}$ increases
  only \emph{finitely} many times; i.e., $\{\rho_k\}$ is \emph{bounded above}
  (\cite{He00_1J,He02_1J} also assume bounded $\{\rho_k\}$). Recently,
  Lin et al.~\cite{Ma09_1J} have proposed an \admm~algorithm %with an increasing penalty sequence
  % is proposed
  for solving PCP problem in \eqref{eq:component_pursuit},
  i.e. \eqref{prob:SPCP-missing} with $\delta=0$, and show that the
  algorithm converges for a \emph{nondecreasing} $\{\rho_k\}$ such that
  $\sum_{k=1}^\infty\rho_k^{-1}=\infty$. The analysis in~\cite{Ma09_1J} % is
  % based on the fact that the
  relies on the fact that the  subdifferentials of any norm are
  \emph{uniformly bounded}. When $\delta>0$ in
  \eqref{prob:SPCP-missing}, the
  results in~\cite{Ma09_1J} do not hold because the subdifferentials of
  the objective function  in the \admm~formulation
  \eqref{eq:problem_generic_split} are no longer uniformly bounded because
  of % the we do not have \emph{uniformly bounded}
  % subdifferentials in the \admm~formulation
  % \eqref{eq:problem_generic_split} due to
  the indicator function used to model the constraint.

  % Indeed, this fact is pointed out on page 107
  % in~\cite{He02_1J}: for the analysis in their paper to hold, the authors
  % emphasized that the functions $F$ and $G$ must be continuous (and hence
  % single-valued). Therefore, their approach is different from
  % monotonicity-based analyses like \cite{EcksteinB92,Kont98_1J} that allow
  % for set-valued monotone $F$ and $G$, and provide stronger convergence
  % guarantees.

% We believe
%   that this is the first result of its kind in
% the literature on the convergence of variable penalty \admm~when
% $\rho_k\nearrow\infty$ and the sub-differential of the non-smooth
% objective function is not uniformly bounded, as is the case for the
% equivalent formulation \eqref{eq:problem_generic_split} that we solved
% using \admip.}
% \eqref{eq:problem_generic_split}.
% -\admip~solves
% \eqref{eq:problem_generic_split} that is equivalent to
% \eqref{prob:SPCP-missing}.}

%Although at first  NSA displayed in Figure~\ref{alg:nsa} may appear to be simply an application of the alternating direction method of multipliers~(ADMM) to the SPCP problem. This, however, is not true.
  In \admm~algorithms~\cite{Boyd11_1J,Eckstein12,EcksteinB92}, the penalty
  parameter % $\rho_k$
  is typically held constant, i.e. $\rho_k = \rho>0$, for all $k\geq 1$.
  %This is motivated by the result
  Although convergence is guaranteed for all $\rho>0$, the empirical
    performance of \admm~algorithms is critically   dependent on the choice of
    penalty parameter $\rho$ -- it deteriorates very rapidly if the penalty
    is set too large or too
    small~\cite{Fukushima92_1J,Glowinski00_1B,Kont98_1J}. Moreover, it is discussed in~~\cite{Lions-Mercier-79} that there exists
    %an optimal
    % multiplier value
    $\rho^\ast$ which optimizes the convergence
    rate for the constant penalty \admm~scheme; %  that there exists
    % $\rho^\ast>0$ which optimizes the speed of convergence bounds.
    however, estimating $\rho^\ast$ is difficult in
    practice~\cite{He00_1J}.
   % , the number of
  % iterations required for Furthermore, it
  % is discussed
  % in~\cite. Indeed, empirical results
  % from applications~
  %  have
  % shown that when $\rho_k=\rho$ is chosen inappropriately, too large or too
  % small, the number of iterations required can increase dramatically.
  % Boyd et al.~\cite{Boyd11_1J} refer the reader to the work of Rockafellar
  % \cite{Rock76_1J} for the case where the multipliers change in every
  % iteration; however, in \cite{Rock76_1J} the sequence multipliers is
  % $\rho_k = 1/c_k$, for $c_k>c>0$ for all $k>1$, i.e. the sequence is
  % bounded above by a fixed constant $1/c$.  He and Yang~\cite{He98_1J}
  % and  Kontogiorgis and Meyer~\cite{Kont98_1J}
  % also show convergence of ADMM assuming a bounded multiplier sequence.
  % In our paper, we show that both primal and dual iterate sequences
  % computed by \admip~converge when $\{\rho_k\}_{k\in\integers_+}$ is
  %   monotonically increasing and unbounded.
  %   In this paper, we show that \admip~converges for an unbounded
  % sequence of penalty multipliers.

The main advantages of adopting an increasing sequence of penalties are as follows:
  \begin{enumerate}[(i)]
  \item The algorithm is robust in the sense that there is no need to
    search for an optimal $\rho^*$.
  \item The algorithm is likely to achieve primal feasibility
    faster. \admm~algorithms can
    be viewed as inexact variant of augmented Lagrangian algorithms where
    one updates the dual iterate after all primal iterates are updated by
    taking a single block-coordinate descent step in each block. The
    primal infeasibility in
    augmented Lagrangian methods can be % given by
% For the augmented
%     Lagrangian methods, the
%   primal infeasibility at the $k$-th iteration is
  approximated by
  $\cO\left(\rho_k^{-1}\norm{Y_k-Y^*}\right)$, where $Y_k$ is an
  estimate of optimal dual $Y^*$ at the $k$-th iteration (see, e.g. Section 17.3
  in~\cite{Nocedal-Wright-99}). %  This approximate bound suggests that
  % convergence rate can be increased by carefully selecting an increasing
  % sequence of penalties.
  Consequently, a suitably chosen increasing sequence of penalties can
  improve the convergence rate.
  \item The complexity of initial (transient)
    iterations can be controlled through controlling the growth in $\{\rho_k\}$. % Starting from a small $\rho_0$ can significantly reduce the
    % computational time in the initial iterations.
  % choosing an increasing sequence:
    The main computational bottleneck in \admip\
    (see Figure~\ref{alg:nsa}) is Step~\ref{algeq:subproblem1}
    that requires an SVD computation
    % of is the
    % bottleneck operation that can
    % be computed as in
    % Computing $L_{k+1}$ requires one to compute an
    % SVD
    (see \eqref{eq:subproblem_L}). Since the optimal $L^*$ is of low-rank,
    and $L_k\rightarrow L^*$, eventually the SVD computations are likely to
    be very efficient. % in practice.
    However, since the initial iterates may
    have large rank, the complexity of the SVD in the initial iterations
    can be quite large. From \eqref{eq:subproblem_L} it follows that
    one does not need to compute singular
    values smaller than $1/\rho_k$; hence,
    starting~\admip~with a \emph{small} $\rho_0>0$ will significantly
    decrease the complexity of initial iterations.
  \end{enumerate}

In this paper, we propose an algorithm that uses an
increasing sequence of penalties. This may appear as a regressive
step that ignores the accumulated numerical experience with penalty and
augmented Lagrangian algorithms. However, we argue that this
experience does \emph{not} immediately carry over to \admm-type algorithms,
and hence, one should
re-examine the role of increasing penalty parameters.
The reluctance to use increasing penalty sequence goes back and is associated with the experience
of solving convex optimization
problems of the form $ P \equiv \min_x\{f(x):\ Ax=b\}$ % . Quadratic p
% enalty methods solves $(P)$ via
using quadratic penalty methods~(QPM). These methods
solve $P$
by  inexactly solving a sequence of subproblems $P_k \equiv
\min_x\{f(x)+\rho_k\norm{Ax-b_k}_2^2\}$ with $b_k=b$ for all $k\geq
1$. Let $x_k$ denote an
 inexact minimizer of $P_k$ %computed by this scheme
 such that the violation in the optimality conditions is
 within a specified tolerance. Then the infeasibility
 $\norm{Ax_k-b}_2$ is $\cO(\frac{1}{\rho_k})$;
% that is a decreasing function of the
% iteration count
%  $k$.
% }
% For this scheme, the
% constraint violation in the $k$-th iteration,
% in iterative algorithms that penalize the infeasibility using a
% quadratic penalty decreases as
% $\norm{Ax_k-b}_2 = \cO(\rho_k^{-1})$;
therefore, the penalty parameter
$\rho_k$ must
be increased to infinity in order to ensure feasibility.
 % in the limit for the penalty methods.
Traditionally, each inexact solution $x_k$ is computed using a second-order
method where the Hessian is of the form $\nabla^2 f(x)+2\rho_kA^TA$. It is important to note that since
the condition number is an increasing function of $\rho_k$,
one encounters numerical instabilities while solving $P_k$ for large $k$ values.
% a second-order method is used to compute the inexact solution $x_k$ to $(P_k)$, followed by an increase in the penalty parameter. On the other hand, the condition number of the Hessian is an increasing function of the penalty parameter; therefore, when the penalty parameter $\rho_k$ is large, one encounters numerical instabilities
% %, and also an increase in the number of iterations required to
% while solving the subproblem $(P_k)$. %  corresponding
% % to a particular penalty value.
On the other hand, in augmented Lagrangian methods~(ALM), i.e. \emph{method of multipliers},
one computes an inexact solution $x_k$ to the
subproblem $P_k$ %\equiv  \min_x\{f(x)+\rho_k\norm{Ax-b -y_k}_2^2\}$ % where
with $b_k=b+y_k$, and then
updates $y_{k+1} = \frac{\rho_k}{\rho_{k+1}}(b_k-Ax_k)$, for all $k\geq 1$. In contrast to QPM, ALM
guarantees primal convergence for a constant penalty sequence, i.e. $\rho_k=\rho$ for all $k\geq 1$;
% can work with a constant penalty,
% i.e. $\rho_k=\rho>0$ for all $k\geq 1$.
% In each iteration of ALM ALM iterates $x_k$ are
%   guaranteed to converge to a primal optimal solution
hence, obviating the need to choose an increasing penalty sequence, and avoiding the numerical instability encountered while solving $P_k$ for large $k$.
% One very
%   important property of ALM is that it ensures convergence to an optimal
%     solution of $(P)$ using a fixed penalty parameter $\rho>0$, hence,
%     obviating the need for $\rho_k$ to grow unboundedly in quadratic
%     penalty methods.
    In this context, proposing an algorithm, \admip, that uses an
increasing sequence of penalties would appear to be contradictory, ignoring the accumulated numerical experience with penalty and augmented Lagrangian algorithms. % On the other hand, we believe that
However, this experience does
not immediately carry over to \admm-type algorithms; % and hence, one should
% re-examine the role of increasing penalty parameters.
there are significant differences between
\admip\ and the quadratic penalty methods, that suggest that the numerical issues
observed in penalty methods are not likely to arise in \admip, and
therefore, an increasing sequence of penalties is worth
revisiting. Indeed, \admip~is a \emph{first-order} algorithm that only employs
shrinkage~\cite{Daubechies-Defrise-DeMol-04} type operations in each
iteration (see Step~\ref{algeq:subproblem1} and Step~\ref{algeq:subproblem2} of
\admip~displayed in Figure~\ref{alg:nsa}). Moreover, unlike quadratic
penalty methods that solve  the subproblems $P_k$ to an accuracy that
  increases with $k$, \admip~takes only one step for each
$P_k$; more importantly, each step can be computed in closed form and is not prone to numerical
instability; thus, avoiding the numerical
problems associated with quadratic penalty methods due to use of an
increasing penalty sequence. Furthermore, the results of our
numerical experiments reported in Section~\ref{sec:computations} clearly
indicate that using an increasing sequence
of penalty multipliers results in faster convergence in practice; in fact, the
performance of \admip\ dominates the performance of \admm-type algorithms
for any fixed penalty term. % The results of our numerical experiments
% comparing \admip~with \proc{ASALM} show that \admip~is faster and also
The numerical experiments also confirm that \admip\ is significantly more robust to changes
in problem parameters. %\red{Before we conclue this section, it is also important to note that, }

\subsection*{Organization}
% This paper is dedicated to developing an efficient algorithm, NSA, for solving SPCP problem with missing data stated in
% \eqref{prob:SPCP-missing}.
We propose \admip~in Section~\ref{sec:nsa} % and state some technical
% lemmas that will be used to
and prove its convergence in
Section~\ref{sec:convergence}. In  Section~\ref{sec:computations} we
report the results of our numerical experiments where we compare
the performance of \admip~with \asalm~on a set of synthetic randomly generated problems and
% we show how our proposed
% method can be applied to solve
on a large-scale problem involving foreground extraction from a noisy surveillance video.
% We report numerical results on background extraction and on some
% synthetic problems showing that NSA is competitive with current
% state-of-the-art solver for SPCP in terms of accuracy and speed.
\vspace*{-0.05in}
\begin{figure}[!ht]
    \rule[0in]{6.5in}{1pt}\\
    \textbf{Algorithm \admip($Z_0,Y_0, \{\rho_k\}_{k\in\integers_+}$)}\\
    \rule[0.125in]{6.5in}{0.1mm}
    \vspace{-0.25in}
    {\footnotesize
    \begin{algorithmic}[1]
    \STATE \textbf{input:} $Z_0\in\reals^{m\times n}$, $Y_0\in\reals^{m\times n}$, $\{\rho_k\}_{k\in\integers_+}\subset\reals_{++}$ such that $\rho_{k+1}\geq\rho_k$, $\rho_k\rightarrow\infty$
    \STATE $k \gets 0$
    \WHILE{$k\geq 0$}
    \STATE $L_{k+1}\gets\argmin_L\{\norm{L}_*+\fprod{Y_k, L-Z_k}+\frac{\rho_k}{2}\norm{L-Z_k}_F^2\}$ \label{algeq:subproblem1}
    %\STATE $\hat{Y}_{k+1}\gets Y_k+\rho_k (L_{k+1}-Z_{k})$
    \STATE $(Z_{k+1},S_{k+1})\gets\argmin_{\{(Z,S): \norm{\proj{Z+S-D}}_F\leq\delta\}}\left\{\xi\norm{S}_1+\fprod{-Y_k, Z-L_{k+1}}+\frac{\rho_k}{2}\norm{Z-L_{k+1}}_F^2\right\}$ \label{algeq:subproblem2}
    %\STATE Compute $\theta_{k}$ - an optimal Lagrangian dual variable for the $\frac{1}{2}\norm{\proj{Z+S-D}}^2_F\leq\frac{\delta^2}{2}$ constraint
    \STATE $Y_{k+1}\gets Y_k+\rho_k (L_{k+1}-Z_{k+1})$
    %\STATE Choose $\rho_{k+1}$ such that $\rho_{k+1}\geq\rho_{k}$
    \STATE $k \gets k + 1$
    \ENDWHILE
    \end{algorithmic}
    }
    \rule[0.125in]{6.5in}{0.1mm}
    \vspace*{-0.4in}
    \caption{\admip: Alternating Direction Method with Increasing Penalty}
    \label{alg:nsa}
\end{figure}
%     }
% \end{algorithm}
% In this section we present the algorithm, NSA, for solving problem
% \eqref{prob:SPCP-missing} that is based on partial variable
% splitting combined with alternating minimization of augmented
% Lagrangian functions with increasing penalty multipliers.
\vspace*{-0.1in}
\section{An \admm~algorithm with partial variable splitting and increasing penalty sequence}
\label{sec:nsa}
Let
\eq
\chi:=\{(Z,S)\in\reals^{m\times n}\times\reals^{m\times
  n}:~\norm{\proj{Z+S-D}}_F\leq\delta\}
\en
denote the feasible set in
\eqref{prob:SPCP-missing} and let $\mathbf{1}_\chi(\cdot,\cdot)$ denote the
indicator function of the closed convex set
$\chi\subset\reals^{m\times n}\times\reals^{m\times n}$, i.e. if
$(Z,S)\in\chi$, then $\mathbf{1}_\chi(Z,S)=0$; otherwise,
$\mathbf{1}_\chi(Z,S)=\infty$.
We use partial variable splitting, i.e. we only split the $L$
variables in \eqref{eq:stable_component_pursuit}, to arrive at the
following equivalent problem
\begin{align}
\label{eq:problem_generic_split}
\min_{L,Z,S\in\reals^{m\times n}}\{\norm{L}_*+\xi~\norm{S}_1+\mathbf{1}_\chi(Z,S):\ L=Z\}.
\end{align}
The augmented Lagrangian function of \eqref{eq:problem_generic_split} is defined as follows:
\begin{align}
\label{eq:augmented_lagrangian}
\mathcal{L}_\rho(L,Z,S;Y)=\norm{L}_*+\xi~\norm{S}_1+\mathbf{1}_\chi(Z,S)+\fprod{Y, L-Z}+\frac{\rho}{2}\norm{L-Z}_F^2.
\end{align}
%\red{Then minimizing \eqref{eq:augmemted_lagrangian} by alternating between
%$L$ and then $(Z,S)$ subject to $(Z,S)\in\chi$, updating the Lagrangian multiplier $Y$, and increasing the penalty multiplier $\rho$
%leads to one iteration of Algorithm NSA displayed in Figure~\ref{alg:nsa}}.
In each iteration of \admip~in Figure~\ref{alg:nsa}, the next iterate $L_{k+1}$ is computed by minimizing \eqref{eq:augmented_lagrangian} over $L\in\reals^{m\times n}$ by setting $\rho=\rho_k$ and $(Y,Z,S)=(Y_k, Z_k, S_k)$; the next iterate $(Z_{k+1},S_{k+1})$ is computed by minimizing \eqref{eq:augmented_lagrangian} over $(Z,S)\in\chi$, by setting $\rho=\rho_k$ and $(Y,L)=(Y_k, L_{k+1})$; finally we set the next dual variable $Y_{k+1}=Y_k+\rho_k(L_{k+1}-Z_{k+1})$.

The computational complexity of each iteration of \admip~is
determined by the subproblems solved in Step~\ref{algeq:subproblem1}
and Step~\ref{algeq:subproblem2}.
%It is crucial to solve the subproblems given in Step~\ref{algeq:subproblem1} and Step~\ref{algeq:subproblem2} in \textbf{Algorithm~\ref{alg:nsa}} efficiently.
The subproblem in Step~\ref{algeq:subproblem1} is a matrix
shrinkage problem and can be solved efficiently by computing an SVD
%singular value decomposition~(SVD)
of an $m\times n$ matrix. The explicit solution of the matrix shrinkage problem is given in
\eqref{eq:subproblem_L}. The subproblem in Step~\ref{algeq:subproblem2}
has the following generic form:
\begin{align}
\label{eq:subproblem_nsa}
(P_{ns}):\  \min\left\{\xi\norm{S}_1+\left\langle Q, Z-\tilde{Z}\right\rangle+\frac{\rho}{2}\norm{Z-\tilde{Z}}_F^2:\ (Z,S)\in \chi\right\},
\end{align}
where $\rho>0$, $Q$, $\tilde{Z}\in\reals^{m\times n}$ are given
problem parameters. %  Lemma~\ref{lem:subproblem} shows that this
                    %  computation can be done efficiently.
% Now before stating Lemma~\ref{lem:subproblem}, we would like to fix
% the notation we will be using throughout the paper. Let $\cE$ be a
% metric space, $A\subset\cE$ and $f:\cE\rightarrow\reals$. $f(A)$
% denotes the image of $A$ and $\partial f(x)$ denotes the set of
% subgradients of $f$ at $x\in\cE$.

\begin{lemma}
\label{lem:subproblem}
The optimal solution $(Z^*,S^*)$ to problem $(P_{ns})$ can be written
in closed form.
\begin{enumerate} [(i)]
\item Suppose $\delta>0$. Then
  \begin{eqnarray}
    S^* & = &
    \sgn\left(\proj{D-q(\tilde{Z})}\right) \odot\max\left\{\left|\proj{D-q(\tilde{Z})}\right|-\xi\frac{(\rho+\theta^*)}{\rho\theta^*}~E,\
      \mathbf{0}\right\}, \label{lemeq:S}\\
     Z^* & = &  \proj{\frac{\theta^*}{\rho+\theta^*}~(D-S^*)+\frac{\rho}{\rho+\theta^*}~q(\tilde{Z})}+\pi_{\Omega^c}\left(q(\tilde{Z})\right), \label{lemeq:Z}
  \end{eqnarray}
  where $q(\tilde{Z}):=\tilde{Z}-\rho^{-1}~Q$; $E$ and
  $\mathbf{0}\in\reals^{m\times n}$ are matrices with all components
  equal to ones and zeros, respectively; $\odot$ denotes the
  component-wise multiplication operator. When
  $\norm{\pi_\Omega(D-q(\tilde{Z}))}_F\leq\delta$, the multiplier $\theta^*=0$; otherwise, $\theta^*$ is the unique positive solution of the nonlinear equation $\phi(\theta)=\delta$, where
  \begin{align}
    \label{eq:phi-def}
    \phi(\theta):= \norm{\min\left\{\frac{\xi}{\theta}~E,\ \frac{\rho}{\rho+\theta}~\left|\proj{D-q(\tilde{Z})}\right|\right\}}_F.
  \end{align}
  The multiplier $\theta^*$ can be efficiently computed in $\cO(|\Omega|\log(|\Omega|))$ time.
\item Suppose $\delta=0$. Then
  \begin{equation}
    \label{lemeq:LS_nonsmooth_delta0}
    S^*=\sgn\left(\proj{D-q(\tilde{Z})}\right) \odot\max\left\{\left|\proj{D-q(\tilde{Z})}\right|-\xi\rho^{-1}~E,\ \mathbf{0}\right\},
  \end{equation}
  and $Z^*=\proj{D-S^*}+\projc{q(\tilde{Z})}$.
\end{enumerate}
\end{lemma}
\begin{proof}
Proof is almost the same with that of Lemma~6.1 in~\cite{Aybat13_1j}. For the sake of completeness, we included the proof in Appendix~\ref{app:proof-1}.
\end{proof}

Note that Lemma~\ref{lem:subproblem} also gives the  worst case
computational complexity of proximal
gradient type first-order methods such as FISTA~\cite{Beck09_1J} and
Algorithm~2 in \cite{Tseng08} applied to the ``smoothed'' version of
the SPCP problem
 % Instead of directly solving $(P)$, FISTA~\cite{Beck09_1J} and
 % Algorithm~2 in \cite{Tseng08} can be applied to the problem:
% \begin{align}
% \label{eq:problem_ns}
$\min_{L,S\in\reals^{m\times n}}\{f_\mu(L)+\xi~\norm{S}_1:\ (L,S)\in\chi\}$,
where % $f_\mu(.)\in C^{1,1}$ is a smooth approximation of
      % $\norm{.}_*$, which is defined as follows
% \begin{align}
$f_\mu(L)=\max_{U\in\reals^{m\times n}:\norm{U}_2\leq
  1}\fprod{L,U}-\frac{\mu}{2}\norm{U}_F^2$. For $\mu=\Theta(\epsilon)$, Lemma~\ref{lem:subproblem} implies that FISTA computes an $\epsilon$-optimal solution of problem $\eqref{prob:SPCP-missing}$ in $\cO(1/\epsilon)$ iterations.
% iterations.% \label{eq:smooth_f}
% \end{align}
% Since $\grad f_\mu(.)$ is Lipschitz continuous with constant
% $L_\mu=\frac{1}{\mu}$, FISTA converges to the optimal solution of
% \eqref{eq:problem_ns} with the convergence rate of
% $\cO\left(\frac{\mu^{-1}}{k^2}\right)$. Thus, setting
% $\mu=\Omega(\epsilon)$, FISTA computes an $\epsilon$-optimal,
% feasible solution of problem $(P)$ in $k^*=\cO(1/\epsilon)$
% iterations.
% \end{remark}

The following lemma will be used later in
Section~\ref{sec:convergence}. However, we state it here since it is related to problem~$(P_{ns})$.
%uses some of the notation from Lemma~\ref{lem:subproblem}.
\begin{lemma}
\label{lem:chi_subgradient}
Suppose that $\delta>0$. Let $(Z^*,S^*)$ be an optimal solution to problem~$(P_{ns})$
%\begin{equation}
%\label{lemeq:subproblem}
%\begin{array}{ll}
%\min_{L,S\in\reals^{m\times n}} & \xi\norm{S}_1+\fprod{Q,L-\tilde{L}}+\frac{\rho}{2}\norm{L-\tilde{L}}_F^2,\\
%s.t. & \frac{1}{2}~\norm{L+S-D}^2_F\leq\frac{1}{2}~\delta^2,
%\end{array}
%\end{equation}
and $\theta^*$ be an optimal Lagrangian multiplier such that
$(Z^*,S^*)$ and $\theta^*$ together satisfy the Karush-Kuhn-Tucker~(KKT) conditions. Then $(W^*,W^*)\in\partial \mathbf{1}_\chi(Z^*,S^*)$, where %$\chi=\{(L,S)\in\reals^{m\times n}:\ \norm{L+S-D}_F\leq\delta\}$ and
$W^*:=-Q+\rho(\tilde{Z}-Z^*)=\theta^*~\proj{Z^*+S^*-D}$.
\end{lemma}
\begin{proof}
See Appendix~\ref{app:proof-2} for the proof.
\end{proof}

\section{Convergence of \admip}
\label{sec:convergence}
When $\rho_k=\rho>0$ for all $k\geq 1$, the convergence of \admip~directly follows from the standard convergence theory of \admm~-see a recent survey paper \cite{Boyd11_1J} for the proof of convergence. In the rest of the paper, we will focus on the case where $\{\rho_k\}_{k\in\integers_+}$ is a monotonically increasing sequence, and we prove that \admip~primal-dual iterate sequence $\{(L_k,S_k,Y_k)\}_{k\in\integers_+}$ converges under mild conditions on the penalty sequence $\{\rho_k\}_{k\in\integers_+}$. We first
establish a sequence of results that extend the similar results in
\cite{Ma09_1J} to the case of constrained subproblems and partial splitting of
variables.
%Throughout this section, other than the NSA iterate sequence $\{L_k,Z_k,S_k,Y_k\}_{k\in\integers_+}$, we will use two other sequences: $\{\hat{Y}_k\}_{k\in\integers_+}$ and $\{\theta_k\}_{k\in\integers_+}$.}
%\red{$\{\hat{Y}_k\}_{k\in\integers_+}$ is defined as follows:}
Define $\{\hat{Y}_k\}_{k\in\integers_+}$ as
\begin{align}
\label{eq:yhat}
\hat{Y}_{k+1}:=Y_k+\rho_k(L_{k+1}-Z_k).
\end{align}
%$\{\hat{Y}_k\}_{k\in\integers_+}$  is not actually computed in NSA and it is an auxiliary sequence which is only used in the theoretical results stated in this section.

The subproblem in Step~\ref{algeq:subproblem2} of \admip is equivalent to
\begin{align}
\min_{Z,S}\left\{\xi\norm{S}_1+\fprod{-Y_k, Z-L_{k+1}}+\frac{\rho_k}{2}\norm{Z-L_{k+1}}_F^2:\ \frac{1}{2}\norm{\proj{Z+S-D}}^2_F\leq\frac{\delta^2}{2}\right\}.
\label{eq:subproblem2_equiv}
\end{align}
In Lemma~\ref{lem:subproblem} we show that the optimal solution of this problem can be written in closed form in terms of $\theta^*$ such that $\phi(\theta^*)=\delta$. Let $\theta_k$ denote the value of $\theta^*$ when Lemma~\ref{lem:subproblem} is applied to the instance in \eqref{eq:subproblem2_equiv}. Then the proof of Lemma~\ref{lem:subproblem} implies that $\theta_k$ is the optimal dual corresponding to the constraint in \eqref{eq:subproblem2_equiv}.

%\red{On the other hand, for all $k\in\integers_+$, $\theta_k$ is actually computed while solving the $k$-th subproblem in Step~\ref{algeq:subproblem2} of NSA and $\theta_k$ is equal to $\theta^*$ given in the statement of Lemma~\ref{lem:subproblem} applied to the $k$-th subproblem. Moreover, the proof of Lemma~\ref{lem:subproblem} shows that $\theta_k$ computed that way is an optimal Lagrangian dual variable for the $(Z,S)\in\chi$ constraint, i.e. $\theta_k$ satisfies
%\begin{align}
%\label{eq:theta_def}
%(Z_{k+1},S_{k+1})=\argmin_{Z,S}\left\{\xi\norm{S}_1+\fprod{-Y_k, Z-L_{k+1}}+\frac{\rho_k}{2}\norm{Z-L_{k+1}}_F^2+\frac{\theta_k}{2}\left(\norm{\proj{Z+S-D}}^2_F-\delta^2\right)\right\}.
%\end{align}}
\begin{lemma}
\label{lem:subgradients}
Let $f(\cdot):=\norm{\cdot}_*$, $g(\cdot):=\xi~\norm{\cdot}_1$ and let
$\{L_k,Z_k,S_k,Y_k\}_{k\in\integers_+}$ denote the \admip~iterates corresponding to the penalty sequence $\{\rho_k\}_{k\in\integers_+}$ and let $\{\hat{Y}_k\}_{k\in\integers_+}$ denote the sequence defined in \eqref{eq:yhat}. Then for
all $k\geq 1$, $-Y_k\in\partial g(S_k)$ and
$-\hat{Y}_k\in\partial
f(L_k)$. Thus,
$\{Y_k\}_{k\in\integers_+}$ and $\{\hat{Y}_k\}_{k\in\integers_+}$ are
bounded sequences. Moreover, $\proj{Y_k}=Y_k$ for all $k\geq 1$.
\end{lemma}
\begin{proof}
See Appendix~\ref{app:proof-3} for the proof.
\end{proof}

Before discussing the convergence properties of \admip~in Theorem~\ref{thm:main}, we need to state a technical result in Lemma~\ref{lem:finite_sums} which will play a key role in proving the main result of this paper: Theorem~\ref{thm:main}.
\begin{lemma}
\label{lem:finite_sums}
Suppose $\delta>0$. Let
$\{L_k,Z_k,S_k,Y_k\}_{k\in\integers_+}$ denote the \admip~iterates corresponding to the non-decreasing sequence of penalty multipliers, $\{\rho_k\}_{k\in\integers_+}$. Let
$(L^*,L^*,S^*)\in\argmin_{L,Z,S}\{\norm{L}_*+\xi~\norm{S}_1:\
\frac{1}{2}\norm{\proj{Z+S-D}}^2_F\leq\frac{\delta^2}{2},\ L=Z\}$ denote
any optimal solution,  $Y^*\in\reals^{m\times n}$ and $\theta^*\geq 0$
denote any optimal Lagrangian duals corresponding to the constraints $L=Z$
and  $\frac{1}{2}\norm{\proj{Z+S-D}}^2_F\leq\frac{\delta^2}{2}$,
respectively. Then
$\{\norm{Z_{k}-L^*}_F^2+\rho_{k}^{-2}\norm{Y_{k}-Y^*}_F^2\}_{k\in\integers_+}$
is a non-increasing  sequence and
\begin{equation*}
\begin{array}{ll}
\sum_{k\in\integers_+}\norm{Z_{k+1}-Z_k}_F^2<\infty,\hspace{5mm}
&\sum_{k\in\integers_+}\rho_{k}^{-2}\norm{Y_{k+1}-Y_k}_F^2<\infty,\\
\sum_{k\in\integers_+}\rho_k^{-1}\fprod{-Y_{k+1}+Y^*,
  S_{k+1}-S^*}<\infty,\hspace{5mm}
&\sum_{k\in\integers_+}\rho_k^{-1}\fprod{-\hat{Y}_{k+1}+Y^*,
  L_{k+1}-L^*}<\infty,
\end{array}
\end{equation*}
\vspace{-5mm}
\begin{equation*}
\begin{array}{c}
\sum_{k\in\integers_+}\rho_k^{-1}\fprod{Y^*-Y_{k+1}, L^*+S^*-Z_{k+1}-S_{k+1}}<\infty.
\end{array}
\end{equation*}
\end{lemma}
\begin{proof}
See Appendix~\ref{app:proof-4} for the proof.
\end{proof}

The partial split formulation~\eqref{eq:problem_generic_split} is equivalent to
$$\min_{L,Z,S\in\reals^{m\times n}}\left\{\norm{L}_*+\xi~\norm{S}_1:\ L=Z,\ \frac{1}{2}\norm{\proj{Z+S-D}}^2_F\leq\frac{\delta^2}{2}\right\}.$$
The Lagrangian function for this formulation is given by
\begin{align}
\label{eq:lagrangian_split}
\mathcal{L}(L,Z,S;Y,\theta)=\norm{L}_*+\xi~\norm{S}_1+\fprod{Y,L-Z}+\frac{\theta}{2}\left(\norm{\proj{Z+S-D}}_F^2-\delta^2\right).
\end{align}
\begin{theorem}
\label{thm:main}
Suppose $\delta>0$. Let
$\{L_k,Z_k,S_k,Y_k\}_{k\in\integers_+}$ denote the \admip~iterates corresponding to the penalty multiplier sequence $\{\rho_k\}_{k\in\integers_+}$. Let $\{\theta_k\}_{k\in\integers_+}$ be the sequence such that $\theta_k$ is the optimal dual corresponding to the constraint in \eqref{eq:subproblem2_equiv}.
\begin{enumerate}[(i)]
\item Suppose $\{\rho_k\}_{k \in\integers_+}$ is a non-decreasing sequence
  such that $\sum_{k\in\integers_+}\frac{1}{\rho_k}=\infty$. Then
  $L^*:=\lim_{k\in\integers_+}Z_k=\lim_{k\in\integers_+}L_k$ and
  $S^*:=\lim_{k\in\integers_+}S_k$ exist; and  $(L^*,
  S^*)$ are optimal for the SPCP problem. % \in\argmin\{\norm{L}_*+\xi~\norm{S}_1:\
  % \norm{\proj{L+S-D}}_F\leq\delta\}$.
\item Suppose $\{\rho_k\}_{k\in\integers_+}$ is a non-decreasing sequence
  such that $\sum_{k\in\integers_+}\frac{1}{\rho_k^2}=\infty$. Then,
  in the case that $\norm{\proj{D-L^*}}_F\neq \delta$,
  $(Y^\ast,\theta^*):=\lim_{k\in\integers_+}(Y_k,\theta_k)$ exists, and
  % $Y^*:=\lim_{k\in\integers_+}Y_k$ exist, and
  $(L^*,L^*,S^*,Y^*,\theta^*)$ is a saddle point of the Lagrangian
  function $\mathcal{L}$ in \eqref{eq:lagrangian_split}.
  Otherwise, i.e. when $\norm{\proj{D-L^*}}_F=\delta$,
  $\{Y_k,\theta_k\}_{k\in\integers_+}$ has a limit point
  $(Y^*,\theta^*)$, such that $(Y^*,\theta^*)\in\argmax_{Y,\theta}\{$
  $\mathcal{L}(L^*,L^*,S^*;Y,\theta):\ \theta\geq 0\}$.
\end{enumerate}
%and any limit point $(Y^*,\theta^*)$ satisfies
%$(Y^*,\theta^*)=\argmax_{Y,\theta}\{\mathcal{L}(L^*,L^*,S^*,Y,\theta):\
%\theta\geq 0\}$.
\end{theorem}
The condition
$\sum_{k\in\integers_+}\frac{1}{\rho_k}=\infty$ is similar to the
condition in Theorem~2 in \cite{Ma09_1J} that is needed to show that
algorithm I-ALM converges to an optimal solution of the robust PCA
problem.
Let $\Omega=\{(i,j): 1\leq i\leq m,\ 1\leq j\leq n\}$, and
$D=L^0+S^0+N^0$ be given such that % $\norm{N^0}_F\leq\delta$ and
$(L^0,S^0,N^0)$ satisfies the assumptions of
Theorem~\ref{thm:candes2} and $\norm{S^0}_F>\sqrt{Cmn}\delta$. Then,
with very high probability,
$\norm{D-L^*}_F>\delta$, where $C$ is the numerical constant defined
in Theorem~\ref{thm:candes2}. Therefore, in practice,  one is unlikely
to encounter the case where
$\norm{D-L^*}_F=\delta$.

\begin{proof}
Lemma~\ref{lem:finite_sums} and the fact that
$L_{k+1}-Z_{k+1}=\frac{1}{\rho_k}~(Y_{k+1}-Y_k)$ for all $k\geq 1$,
together imply that
\begin{align*}
\infty>\sum_{k\in\integers_+}\rho_{k}^{-2}\norm{Y_{k+1}-Y_k}_F^2
=\sum_{k\in\integers_+}\norm{L_{k+1}-Z_{k+1}}_F^2.
\end{align*}
Thus, $\lim_{k\in\integers_+}(L_k-Z_k)=0$.

Let
$(L^\#,L^\#,S^\#)\in\argmin_{L,Z,S}\{\norm{L}_*+\xi~\norm{S}_1:\
\frac{1}{2}\norm{\proj{Z+S-D}}^2_F\leq\frac{\delta^2}{2},\
L=Z\}$ denote any optimal solution, $Y^\#\in\reals^{m\times n}$ and
$\theta^\#\geq 0$ denote any Lagrangian dual optimal solutions corresponding to
$L=Z$ and $\frac{1}{2}\norm{\proj{Z+S-D}}^2_F\leq\frac{\delta^2}{2}$
constraints, respectively, and $f^*:=\norm{L^\#}_*+\xi~\norm{S^\#}_1$.

%Moreover, let $\chi=\{(Z,S)\in\reals^{m\times n}\times\reals^{m\times
%n}:\ \norm{Z+S-\pi_\Omega\left(D\right)}_F\leq\delta\}$ and
%$\mathbf{1}_\chi(Z,S)$ denote the indicator function of the closed
%convex set $\chi$, i.e. $\mathbf{1}_\chi(Z,S)=0$ if $(Z,S)\in\chi$;
%otherwise, $\mathbf{1}_\chi(Z,S)=\infty$.
Since $(Z_k,S_k) \in \chi$ for all $k \geq 1$, or equivalently
$\mathbf{1}_\chi(Z_k,S_k)=0$ for all $k\geq 1$, it follows that
\begin{eqnarray}
\lefteqn{\norm{L_k}_*+\xi~\norm{S_k}_1} \nonumber\\
& =&\norm{L_k}_*+\xi~\norm{S_k}_1 + \mathbf{1}_\chi(Z_k,S_k),
\nonumber \\
& \leq &\norm{L^\#}_*+\xi~\norm{S^\#}_1 + \mathbf{1}_\chi(L^\#,S^\#)
+\fprod{\hat{Y}_k, L^\#-L_k}+\fprod{Y_k, S^\#-S_k}-\fprod{Y_k,
  L^\#+S^\#-Z_k-S_k}, \nonumber \\
& =& f^* + \fprod{-\hat{Y}_k+Y^\#, L_k-L^\#}+\fprod{-Y_k+Y^\#,
  S_k-S^\#}+\fprod{Y^\#-Y_k, L^\#+S^\#-Z_k-S_k} \nonumber\\
&&\mbox{} +\fprod{Y^\#, Z_k-L_k}, \label{eq:convexity_bound}
\end{eqnarray}
where the inequality follows from Lemma~\ref{lem:subgradients} and the
fact
that  %$-Y_k\in\xi~\partial\norm{S_k}_1$, $-\hat{Y}_k\in\partial\norm{L_k}_*$ and
$(Y_k,Y_k)\in\partial\mathbf{1}_\chi(Z_k,S_k)$ -see Lemma~\ref{lem:chi_subgradient}; and
\eqref{eq:convexity_bound} follows  from rearranging the terms and the fact
that $(L^\#,S^\#)\in\chi$.

From Lemma~\ref{lem:finite_sums}, we have that
$$\sum_{k\in\integers_+}\rho_{k-1}^{-1}\left(\fprod{-\hat{Y}_{k}+Y^\#,
    L_{k}-L^\#}+\fprod{-Y_{k}+Y^\#, S_{k}-S^\#}+\fprod{Y^\#-Y_{k},
    L^\#+S^\#-Z_{k}-S_{k}}\right)<\infty.$$
First consider the case where
$\sum_{k\in\integers_+}\frac{1}{\rho_k}=\infty$. There exists
$\cK\subset\integers_+$ such that
\begin{align}
\lim_{k\in\cK}\left(\fprod{-\hat{Y}_{k}+Y^\#,
    L_{k}-L^\#}+\fprod{-Y_{k}+Y^\#, S_{k}-S^\#}+\fprod{Y^\#-Y_{k},
    L^\#+S^\#-Z_{k}-S_{k}}\right)=0. \label{eq:inner_product_limit}
\end{align}
Therefore, \eqref{eq:convexity_bound}, \eqref{eq:inner_product_limit} and
$\lim_{k\in\integers_+}(Z_k-L_k)=0$ together  imply that
\[
\limsup_{k \in \cK} \norm{L_k}_*+\xi~\norm{S_k}_1 \leq
f^*. % =\norm{L^\#}_*+\xi~\norm{S^\#}_1=\min\{\norm{L}_*+\xi~\norm{S}_1:\
% (L,S)\in\chi\}.
\]
Hence, $\{\norm{L_k}_*+\xi~\norm{S_k}_1\}_{k\in\cK}$ is a bounded
sequence. Therefore, there exists $\cK^*\subset\cK\subset\integers_+$
such that $\{(L_k,S_k)\}_{k\in\cK^*}$ has a limit. Let
$(L^*,S^*):=\lim_{k\in\cK^*}(L_k,S_k)$. Since
$\lim_{k\in\integers_+}(Z_k-L_k)=0$ and $(Z_k,S_k)\in\chi$ for all
$k\geq 1$, we have $(L^*,S^*)=\lim_{k\in\cK^*}(Z_k,S_k)\in\chi$. Taking
the limit of both sides of \eqref{eq:convexity_bound} along $\cK^*$
gives
\eq
\norm{L^*}_*+\xi~\norm{S^*}_1
=\lim_{k\in\cK^*}\norm{L_k}_*+\xi~\norm{S_k}_1\leq
f^*,
\en
and since $(L^*,S^*)\in\chi$, we conclude that
$(L^*,S^*)\in\argmin\{\norm{L}_*+\xi~\norm{S}_1:\ (L,S)\in\chi\}$.

Note that
\eq
(L^*,L^*,S^*) \in \argmin_{L,Z,S}\{\norm{L}_*+\xi~\norm{S}_1:\
\frac{1}{2}\norm{\proj{Z+S-D}}^2_F\leq\frac{\delta^2}{2},\ L=Z\}.
\en
Let $\bar{Y}\in\reals^{m\times n}$ and $\bar{\theta}\geq 0$ denote any
Lagrangian dual optimal solutions corresponding to $L=Z$ and
$\frac{1}{2}\norm{\proj{Z+S-D}}^2_F\leq\frac{\delta^2}{2}$
constraints, respectively. Lemma~\ref{lem:subgradients} implies that
$\{Y_k\}$ is a bounded sequence. Thus, from Lemma~\ref{lem:finite_sums}, it
follows that
$\{\norm{Z_{k}-L^*}_F^2+\rho_{k}^{-2}\norm{Y_{k}-\bar{Y}}_F^2\}_{k\in\integers_+}$
is a bounded, non-increasing sequence, and therefore, has a unique
limit point; hence, every subsequence of this sequence
converges to the same limit.  Combining this result with the
facts that $\lim_{k \in \cK^\ast} Z_k = L^\ast$ and $\{Y_k\}_{k\in\integers_+}$ is a bounded sequence, it follows that
\begin{eqnarray*}
\lim_{k\in\integers_+}\norm{Z_{k}-L^*}_F^2 & = &
\lim_{k\in\integers_+}\norm{Z_{k}-L^*}_F^2+\rho_{k}^{-2}\norm{Y_{k}-\bar{Y}}_F^2\\
& = &
\lim_{k\in\cK^\ast}\norm{Z_{k}-L^*}_F^2+\rho_{k}^{-2}\norm{Y_{k}-\bar{Y}}_F^2,\\
& = & \lim_{k\in\cK^\ast}\norm{Z_{k}-L^*}_F^2,\\
& = & 0.
\end{eqnarray*}
% Since every subsequence of a convergent sequence
% converges to the same limit, weget
% \begin{align*}
% \lim_{k\in\integers_+}\norm{Z_{k}-L^*}_F^2=\lim_{k\in\integers_+}\norm{Z_{k}-L^*}_F^2+\rho_{k}^{-2}\norm{Y_{k}-\bar{Y}}_F^2=\lim_{k\in\cK^*}\norm{Z_{k}-L^*}_F^2+\rho_{k}^{-2}\norm{Y_{k}-\bar{Y}}_F^2=0,
% \end{align*}
% where we used  $\lim_{k\in\cK^*}Z_k=L^*$, $\rho_k\nearrow\infty$ as $k\rightarrow\infty$ and $\{\hat{Y}_k\}_{k\in\integers_+}$, $\{Y_k\}_{k\in\integers_+}$ are bounded sequences -see Lemma~\ref{lem:subgradients}.
Since $\lim_{k\in\integers_+}\norm{Z_{k}-L^*}_F=0$ and
$\lim_{k\in\integers_+} (Z_{k}-L_k)=0$, it follows  that $\lim_{k\in\integers_+}L_k=\lim_{k\in\integers_+}Z_k=L^*$.

Lemma~\ref{lem:subproblem} applied to the sub-problem in Step~\ref{algeq:subproblem2} of \admip~corresponding to the $k$-th
iteration gives
\begin{align}
&S_{k+1}=\sgn\left(\proj{D-q(L_{k+1})}\right)\odot\max\left\{\left|\proj{D-q(L_{k+1})}\right|-\xi\frac{(\rho_k+\theta_k)}{\rho_k\theta_k}~E,\
  \mathbf{0}\right\}, \label{eq:Sk}\\
&Z_{k+1}=
\proj{\frac{\theta_k}{\rho_k+\theta_k}~(D-S_{k+1})+\frac{\rho_k}{\rho_k+\theta_k}~q(L_{k+1})}+\pi_{\Omega^c}\left(q(L_{k+1})\right), \label{eq:Lk}
\end{align}
where $q(L_{k+1}):=\left(L_{k+1}+\frac{1}{\rho_k}~Y_k\right)$. Here,
$\theta_k=0$, when $\norm{\proj{D-q(L_{k+1})}}_F\leq\delta$;
otherwise, $\theta_k>0$ is the unique solution of the equation $\phi_k(\theta)=\delta$, where
\begin{equation}
\label{eq:phi_k}
\phi_k(\theta):= \left\|\min\left\{\frac{\xi}{\theta}~E,\
    \frac{\rho_k}{\rho_k+\theta}~\left|\proj{D-q(L_{k+1})}\right|\right\}\right\|_F.
\end{equation}
Since $\lim_{k\in\integers_+}L_k=L^*$, $\{Y_k\}_{k\in\integers_+}$ is
a bounded sequence and $\rho_k\nearrow\infty$, we have that $\lim_{k\in\integers_+}q(L_{k+1})=\lim_{k\in\integers_+}L_{k+1}+\frac{1}{\rho_k}~Y_k=L^*$. Next, we establish $\{S_k\}_{k\in\integers_+}$ has a unique limit point $S^*$.

\begin{enumerate}[(i)]
\item First suppose  $\norm{\proj{D-L^*}}_F\leq\delta$.
  Recall that we have shown that there exists a sub-sequence
  $\cK^*\subset\integers_+$ such that
  \eq
\lim_{k\in\cK^*}(L_k,S_k)=(L^*,S^*)\in\argmin_{L,S}\{\norm{L}_*+\xi\norm{S}_1:\
  \norm{\proj{L+S-D}}_F\leq\delta\}.
  \en
  Since $\norm{\proj{D-L^*}}_F\leq\delta$, $(L^*,\mathbf{0})$ is a
  feasible solution, it follows
  $\norm{L^*}_*+\xi\norm{S^*}\leq\norm{L^*}_*$. Consequently,  $S^*=\mathbf{0}$.
\begin{eqnarray}
\lefteqn{\norm{L_k}_*+\xi~\norm{S_k}_1}\nonumber\\
& =&\norm{L_k}_*+\xi~\norm{S_k}_1 + \mathbf{1}_\chi(Z_k,S_k), \nonumber \\
& \leq &\norm{L^*}_*+\xi~\norm{\mathbf{0}}_1 +
\mathbf{1}_\chi(L^*,\mathbf{0}) %\nonumber \\
%&& \mbox{}
-\fprod{-\hat{Y}_k, L^*-L_k}-\fprod{-Y_k, \mathbf{0}-S_k}-\fprod{Y_k, L^*+\mathbf{0}-Z_k-S_k},\nonumber\\
& = &\norm{L^*}_*+\fprod{\hat{Y}_k, L^*-L_k}+\fprod{Y_k,
  Z_k-L^*}, \label{eq:S_limit_equality_cond}
\end{eqnarray}
where the inequality follows from Lemma~\ref{lem:subgradients} and the
fact that $(Y_k,Y_k)\in\partial \mathbf{1}_\chi(Z_k,S_k)$ (see
Lemma~\ref{lem:chi_subgradient} for details).

Since the sequences $\{Y_k\}_{k\in\integers_+}$ and
$\{\hat{Y}_k\}_{k\in\integers_+}$ are both bounded and
$\lim_{k\in\integers_+}L_k=\lim_{k\in\integers_+}Z_k=L^*$, taking the
limit of  both sides of \eqref{eq:S_limit_equality_cond}, we get
\begin{eqnarray*}
\norm{L^*}_*+\xi~\lim_{k\in\integers_+}\norm{S_k}_1&=&\lim_{k\in\integers_+}\norm{L_k}_*+\xi~\norm{S_k}_1\\
&\leq&\lim_{k\in\integers_+}\norm{L_k}_*+\fprod{\hat{Y}_k,L^*-L_k}+\fprod{Y_k, Z_k-L^*}= \norm{L^*}_*.
\end{eqnarray*}
Therefore, $\lim_{k\in\integers_+}\norm{S_k}_1=0$, which implies that $\lim_{k\in\integers_+}S_k=\mathbf{0}$. Hence, $S^*=\lim_{k\in\integers_+}S_k$.

\item Next, suppose $\norm{\proj{D-L^*}}_F>\delta$.
Since $\lim_{k\in\integers_+}\norm{\proj{D-q(L_{k+1})}}_F =
\norm{\proj{D-L^*}}_F>\delta$, there exists $K\in\integers_+$ such
that for all $k\geq K$, $\norm{\proj{D-q(L_{k+1})}}_F>\delta$. For all
$k\geq K$, $\phi_k(\cdot)$, defined in \eqref{eq:phi_k}, is a
continuous and strictly decreasing
function of $\theta$ for $\theta\geq 0$. Hence, for all $k \geq K$,
the inverse function
$\phi^{-1}_k(.)$ exists in an open neighborhood containing $\delta$. Thus,
$\phi_k(0)=\norm{\proj{D-q(L_{k+1})}}_F>\delta$ for all $k\geq K$ and
$\lim_{\theta\rightarrow\infty}\phi_k(\theta)=0$ imply that
$\theta_k=\phi^{-1}_k(\delta)>0$ for all $k\geq K$. Moreover,
$\phi_k(\theta)\leq\phi(\theta):=\norm{\frac{\xi}{\theta}~E}_F$
implies that for all $k \geq 1$,
\begin{equation}
\label{eq:thetak-upbnd}
\theta_k=\phi_k^{-1}(\delta)\leq\phi^{-1}(\delta)=\frac{\xi\sqrt{mn}}{\delta}.
\end{equation}
Since $\{\theta_k\}_{k\geq K}$ is a bounded sequence, it has a
convergent subsequence $\cK_\theta\subset\integers_+$,
i.e., $\theta^*:=\lim_{k\in\cK_\theta}\theta_k$ exists. We also have
$\phi_k(\theta)\rightarrow\phi_\infty(\theta)$ pointwise for all
$0\leq\theta\leq\frac{\xi\sqrt{mn}}{\delta}$, where
\begin{align}
\phi_\infty(\theta):= \left\|\min\left\{\frac{\xi}{\theta}~E,\
    \left|\proj{D-L^*}\right|\right\}\right\|_F.
\end{align}
Since $\phi_k(\theta_k)=\delta$ for all $k\geq K$, we have
\begin{align}
\delta=\lim_{k\in\cK_\theta}\phi_k(\theta_k)
=\lim_{k\in\cK_\theta}\left\|\min\left\{\frac{\xi}{\theta_k}~E,\
    \frac{\rho_k}{\rho_k+\theta_k}~\left|\proj{D-q(L_{k+1})}
    \right|\right\}\right\|_F=\phi_\infty(\theta^*).
\end{align}
Note that $\phi_\infty(\cdot)$ is also a continuous and strictly
decreasing function of $\theta$ for $\theta\geq 0$. Moreover,
$\phi_\infty(0)=\norm{\proj{D-L^*}}_F>\delta$ implies that
$\phi_\infty$ is invertible around $\delta$, i.e. $\phi_\infty^{-1}$
exists in a neighborhood containing $\delta$, and
$\phi_\infty^{-1}(\delta)>0$. Thus,
$\theta^*=\phi_\infty^{-1}(\delta)$. Since $\cK_\theta$ is an
arbitrary subsequence and $\theta^*=\phi_\infty^{-1}(\delta)$ does not
depend on $\cK_\theta$, we can conclude that
\begin{equation}
  \label{eq:theta-limit}
  \lim_{k\in\integers_+}\theta_k=\phi_\infty^{-1}(\delta)=\theta^*.
\end{equation}
Since
$\theta^*=\lim_{k\in\integers_+}\theta_k$, taking the limit on both
sides of \eqref{eq:Sk}, we get
\begin{align}
S^*:=\lim_{k\in\integers_+}S_{k+1}=\sgn\left(\proj{D-L^*}\right) \odot\max\left\{\left|\proj{D-L^*}\right|-\frac{\xi}{\theta^*}~E,\
  \mathbf{0}\right\},
\end{align}
\end{enumerate}
and this completes the first part of the theorem.

Now, suppose $\{\rho_k\}_{k\in\integers_+}$ is strictly increasing and $\sum_{k =
  1}^\infty \frac{1}{\rho_k^2} = \infty$.  We need two results in
order to establish the convergence of the duals.
From Lemma~\ref{lem:finite_sums}, we have
  $\sum_{k\in\integers_+}\norm{Z_{k+1}-Z_k}_F^2<\infty$. From the
  definition of $\hat{Y}_k$ in (\ref{eq:yhat}), it follows that
  \begin{align}
    \sum_{k\in\integers_+}\rho_k^{-2}\norm{\hat{Y}_{k+1}-Y_{k+1}}_F^2
    = \sum_{k\in\integers_+}\norm{Z_{k+1}-Z_k}_F^2 < \infty.
  \end{align}
  Since $\sum_{k\in\integers_+}\frac{1}{\rho_k^2}=\infty$, there
  exists a sub-sequence $\bar{\cK}\subset\integers_+$ such that
  $\lim_{k\in\bar{\cK}}\norm{\hat{Y}_{k+1}-Y_{k+1}}_F^2=0$. Hence,
  $\lim_{k\in\bar{\cK}}\rho_k^2\norm{Z_{k+1}-Z_k}_F^2=0$,
  i.e.
  \begin{equation}
    \label{eq:Z-limit}
    \lim_{k\in\bar{\cK}}\rho_k(Z_{k+1}-Z_k)=0.
  \end{equation}
  Using \eqref{eq:Lopt_cond}, \eqref{eq:Sopt_cond} and
    \eqref{eq:Zopt_cond} from the proof of
    Lemma~\ref{lem:subgradients} in Appendix~\ref{app:proof-3}, we get
    \begin{eqnarray}
      0 & \in
      & \partial\norm{L_{k+1}}_*+\theta_k\proj{Z_{k+1}+S_{k+1}-D}+\rho_k(Z_{k+1}-Z_k), \label{eq:Lk_opt}\\
      0& \in& \xi\partial\norm{S_{k+1}}_1+ \theta_k\proj{Z_{k+1}+S_{k+1}-D}.\label{eq:Sk_opt}
    \end{eqnarray}
  We will establish the
  convergence of the duals by considering two cases.
  \begin{enumerate}[(i)]
  \item Suppose $\norm{\proj{D-L^*}}_F\neq\delta$. Note that from
    \eqref{eq:Zopt_cond}, it follows that
    $Y_k=\theta_{k-1}\proj{Z_k+S_k-D}$ for all $k\geq 1$.
    First suppose
    that  $\norm{\proj{D-L^*}}_F<\delta$. Since
    \eq
    \lim_{k\in\integers_+}
    \left\|\proj{D-(L_{k+1}+\frac{1}{\rho_k}~Y_k)}\right\|_F
    =\norm{\proj{D-L^*}}_F<\delta,
    \en
    there exists $K\in\integers_+$ such that for all $k\geq K$,
    $\norm{\proj{D-(L_{k+1}+\frac{1}{\rho_k}~Y_k)}}_F<\delta$. Thus,
    from Lemma~\ref{lem:subproblem} for all $k\geq K$, $\theta_k=0$,
    $S_{k+1}=0$, $Z_{k+1}=L_{k+1}+\frac{1}{\rho_k}~Y_k$, which implies
    that $\theta^*=\lim_{k\in\integers_+}\theta_k=0$ and, since $S^*=\lim_{k\in\integers_+}S_k=0$ and
    $\lim_{k\in\integers_+}Z_k=L^*$,
    \eq
    Y^* = \lim_{k\in\integers_+}Y_k =
    \lim_{k\in\integers_+}\theta_{k-1}\proj{Z_{k}+S_{k}-D}=\mathbf{0}.
    \en

    Next, suppose that $\norm{\proj{D-L^*}}_F>\delta$. In this case, we have
    established in (\ref{eq:theta-limit}) that  $\theta^*=\lim_{k\in\integers_+}\theta_k$
    exists. Hence,
    \eq
    \lim_{k\in\integers_+}Y_k=\lim_{k\in\integers_+}\theta_{k-1}
    \proj{Z_k+S_k-D}=\theta^*\proj{L^*+S^*-D} =   Y^\ast.
    \en
    exists.
    % Suppose that $\sum_{k\in\integers_+}\frac{1}{\rho_k^2}=\infty$.

    % We have established above that $\theta^*=\lim_{k\in\integers_+}\theta_k$ and
    % \eq
    % Y^*=\lim_{k\in\integers_+}Y_k=\lim_{k\in\integers_+}
    % \theta_{k-1}\proj{Z_k+S_k-D}=\theta^*\proj{L^*+S^*-D}
    % \en
    % exists when $\norm{D-L^*}\neq\delta$.

    Taking the limit of \eqref{eq:Lk_opt} and \eqref{eq:Sk_opt} along
    $\bar{\cK}\subset\integers_+$ defined in (\ref{eq:Z-limit}); and using the fact that
    $\lim_{k\in\bar{\cK}}\rho_k(Z_{k+1}-Z_k)=0$, we get
    \begin{eqnarray}
      0& \in& \partial\norm{L^*}_*+\theta^*\proj{L^*+S^*-D}, \label{eq:L_opt}\\
      0& \in &\xi\partial\norm{S^*}_1+
      \theta^*\proj{L^*+S^*-D}.\label{eq:S_opt}
    \end{eqnarray}
    % \eqref{eq:L_opt} and \eqref{eq:S_opt} together imply that
    Thus, it follows that the
    primal variables $(L^*, S^*)$ and dual variables
    $Y^*=\theta^*\proj{L^*+S^*-D}$ and $\theta^*$ satisfy KKT optimality
    conditions for the problem \eq
    \min_{L,Z,S}\{\norm{L}_*+\xi~\norm{S}_1:\
    \frac{1}{2}\norm{\proj{Z+S-D}}^2_F\leq\frac{\delta^2}{2},\ L=Z\}.
    \en
    Hence, $(L^*,L^*,S^*,Y^*,\theta^*)$ is a saddle point of the Lagrangian function
    \begin{align*}
      \mathcal{L}(L,Z,S;Y,\theta)
      = \norm{L}_*+\xi~\norm{S}_1
      +\fprod{Y,L-Z}+\frac{\theta}{2}\left(\norm{\proj{Z+S-D}}_F^2-\delta^2\right).
    \end{align*}

  \item Next, consider the case where $\norm{D-L^*}_F=\delta$. Fix
    $k>0$. $\theta_k=0$ if
    $\norm{D-(L_{k+1}+\frac{1}{\rho_k}~Y_k)}_F\leq \delta$; otherwise,
    $\theta_k>0$. Also, from (\ref{eq:thetak-upbnd}) it follows that
    $\theta_k\leq\frac{\xi\sqrt{mn}}{\delta}$. Since
    $\{\theta_k\}_{k\in\integers_+}$ is a bounded sequence, there exists a
    further subsequence $\cK_\theta$ of the sequence $\bar{\cK}$ defined
    in \eqref{eq:Z-limit} such that
    $\theta^*:=\lim_{k\in\cK_\theta}\theta_{k-1}$ and
    $Y^*:=\lim_{k\in\cK_\theta}\theta_{k-1}\proj{Z_k+S_k-D}=\theta^*\proj{L^*+S^*-D}$
    exist.
    Thus, taking the limit of \eqref{eq:Lk_opt},\eqref{eq:Sk_opt}
    along $\cK_\theta\subset\integers_+$ and using the facts that
    $\lim_{k\in\bar{\cK}}\rho_k(Z_{k+1}-Z_k)=0$ and
    $L^*=\lim_{k\in\integers_+}L_k=\lim_{k\in\integers_+}Z_k$,
    $S^*=\lim_{k\in\integers_+}S_k$ exist, we conclude that
    $(L^*,L^*,S^*,Y^*,\theta^*)$ is a saddle point of the Lagrangian
    function $\mathcal{L}(L,Z,S;Y,\theta)$.
  \end{enumerate}
\end{proof}

\section{Numerical experiments}
\label{sec:computations}
We conducted two sets of numerical experiments with \admip~to
solve SPCP problems. In the first set of experiments we solved
randomly generated instances of the SPCP
problem. In this setting, we conducted three different tests.
% performance \admip~ under three different scenarios.}
First,
we compared \admip~with \admm~for different values of the fixed penalty $\rho$;
%~ taking values over an interval};
second, we conducted a set of experiments to understand
how \admip~runtime scales as a function of the problem parameters and
size; and third, we compared
\admip~with \asalm~\cite{Tao09_1J}.
\asalm~is an \admm~algorithm, tailored for the SPCP problem, with a fixed
penalty $\rho$. For each dual update, \asalm~updates \emph{three
  blocks} of primal variables, while \admip~updates \emph{two blocks}.
In the second set of experiments, we compared \admip~and \asalm~ on the
  foreground detection problem, where the goal is to %with the objective of
extract the moving objects from a noisy and corrupted airport security
video~\cite{Li04_1J}. All the numerical experiments were conducted on a
Dell M620 server computing node running on RedHat Enterprise Linux 6 (RHEL
6). Each numerical test was carried out using MATLAB R2013a (64 bit) with
16 GB RAM available on a single core of Intel Leon E5-2665 2.40 GHz
processor. The MATLAB code for \admip\footnote{In an earlier preprint, we named it as Non-Smooth Augmented Lagrangian~\texttt{(NSA)} algorithm.}~is available
at~\url{http://www2.ie.psu.edu/aybat/codes.html} and the code for
\asalm~is available on request from the authors of~\cite{Tao09_1J}.
\subsection{Implementation details}
\begin{figure}[h!]
    \rule[0in]{6.5in}{1pt}\\
    \textbf{Algorithm \admip($Z_0,Y_0, \{\rho_k\}_{k\in\integers_+})$}\\
    \rule[0.125in]{6.5in}{0.1mm}
    \vspace{-0.25in}
    {\footnotesize
    \begin{algorithmic}[1]
    \STATE \textbf{input:} $Z_0\in\reals^{m\times n}$, $Y_0\in\reals^{m\times n}$, $\{\rho_k\}_{k\in\integers_+}\subset\reals_{++}$ such that $\rho_{k+1}\geq\rho_k$, $\rho_k\rightarrow\infty$
    \STATE $k \gets 0$
    \WHILE{$k\geq 0$}
    \STATE Compute $\mathrm{svd}(Z_k-Y_k/\rho_k)$ such that $Z_k-Y_k/\rho_k=U\Diag(\sigma)V^T$
    \STATE $L_{k+1}\gets U\Diag\left(\min\left\{\sigma-\frac{1}{\rho_k}\mathbf{1},0\right\}\right)V^T$ \label{algeq:L-problem}
    \STATE $C\gets L_{k+1}+\rho_k^{-1}Y_k$
    \STATE $\theta^*\gets$\texttt{ThetaSearch}$(|D-C|,\Omega,\delta,\rho_k)$ \label{algeq:theta-problem}
    \STATE $S_{k+1}\gets\sgn\left(\proj{D-C}\right) \odot\max\left\{\left|\proj{D-C}\right|-\xi\frac{(\rho_k+\theta^*)}{\rho_k\theta^*}~E,\ \mathbf{0}\right\}$ \label{algeq:S-problem}
    \STATE $Z_{k+1}\gets \proj{\frac{\theta^*}{\rho_k+\theta^*}~(D-S^*)+\frac{\rho_k}{\rho_k+\theta^*}~C}+\pi_{\Omega^c}\left(C\right)$ \label{algeq:Z-problem}
    \STATE $Y_{k+1}\gets Y_k+\rho_k (L_{k+1}-Z_{k+1})$
    %\STATE Choose $\rho_{k+1}$ such that $\rho_{k+1}\geq\rho_{k}$
    \STATE $k \gets k + 1$
    \ENDWHILE
    \end{algorithmic}
    }
    \rule[0.125in]{6.5in}{0.1mm}
    \vspace*{-0.4in}
    \caption{Pseudocode for \admip}
    \label{alg:pseudocode}
\end{figure}
\begin{figure}[h!]
    \rule[0in]{6.5in}{1pt}\\
    \textbf{Subroutine \texttt{ThetaSearch}($A,\Omega,\delta,\rho$)}\\
    \rule[0.125in]{6.5in}{0.1mm}
    \vspace{-0.25in}
    {\footnotesize
    \begin{algorithmic}[1]
    \STATE \textbf{output:} $\theta^*\in\reals_+$,\ \textbf{input:} $A\in\reals_+^{m\times n}$, $\Omega\subset\{1,\ldots,m\}\times\{1,\ldots,n\}$, $\delta>0$, $\rho>0$
    \IF{$\norm{\proj{A}}_F\leq\delta$}
        \STATE $\theta^*\gets 0$
    \ELSE
        \STATE Compute $0\leq a_{(1)}\leq a_{(2)}\leq \ldots \leq a_{(|\Omega|)}$ by sorting $\{A_{ij}:\ (i,j)\in\Omega\}$
        \STATE $a_{(0)}\gets 0$
        \STATE $\bar{k}\gets\max\{j:\ a_{(j)}\leq\frac{\xi}{\rho},\ 0\leq j\leq |\Omega|\}$
        \IF{$\bar{k}==|\Omega|$}
            \STATE $\theta^*\gets\rho\left(\frac{\norm{\proj{A}}_F}{\delta}-1\right)$
        \ELSE
            \STATE $j^*\gets\bar{k}$
            \FOR{$j=\bar{k}+1,\ldots,|\Omega|$}
                \STATE $\phi_j\gets\sqrt{\left(1-\frac{\xi}{\rho}~a^{-1}_{(j)}\right)^2\sum_{i=0}^j a^2_{(i)}+(|\Omega|-j)\left(a_{(j)}-\frac{\xi}{\rho}\right)^2}$
                \IF{$\phi_j\leq\delta$}
                    \STATE $j^*\gets j$
                \ENDIF
            \ENDFOR
            \IF{$j^*==|\Omega|$}
                \STATE $\theta^*\gets\rho\left(\frac{\norm{\proj{A}}_F}{\delta}-1\right)$
            \ELSE
                \STATE Compute unique $\theta^*>0$ by finding the roots of $\left(\frac{\rho}{\rho+\theta^*}\right)^2\sum_{i=0}^{j^*}a^2_{(i)}+(|\Omega|-1)\left(\frac{\xi}{\theta^*}\right)^2$ \label{algeq:quartic}
            \ENDIF
        \ENDIF
    \ENDIF
    \end{algorithmic}
    }
    \rule[0.125in]{6.5in}{0.1mm}
    \vspace*{-0.4in}
    \caption{\texttt{ThetaSearch}: Subroutine for computing the optimal dual $\theta^*$}
    \label{alg:theta_search}
\end{figure}
The optimal solution of the Step~\ref{algeq:subproblem1} subproblem
corresponding
to the $k$-th iteration is given by
\begin{align}
\label{eq:subproblem_L}
L_{k+1}=U\Diag\left(\min\left\{\sigma-\frac{1}{\rho_k}\mathbf{1},0\right\}\right)V^T,
\end{align}
where $q(Z_k)=Z_k-Y_k/\rho_k=U\Diag(\sigma)V^T$ and $\mathbf{1}$
denotes a vector of all ones. % Note that \eqref{eq:subproblem_L}
                              % requires
Computing the full SVD of $q(Z_k)$ is expensive for large
 instances. However, we do not need to compute the full SVD, because
only the singular values that are larger than $1/\rho_k$ and the
corresponding singular vectors are needed. In order to exploit this
fact, we % adopted
used
a modified version of LANSVD% function of
% PROPACK
~\cite{propack}\footnote{The modified version is available from
  http://svt.stanford.edu/code.html} % % with threshold option
% to compute the partial SVDs. This
% modified version of LANSVD function
% can
that comes with \emph{treshold option} to compute only those singular
vectors with singular values greater than a given
threshold value $\tau>0$. Note that we set $\tau=1/\rho_k$ in
the $k$-th \admip~iteration.

The bottleneck step in the $k$-th iteration of \asalm, which is an
\admm~algorithm with constant penalty $\rho>0$, also involves
computing a low-rank matrix $L_{k+1}$. Indeed, first, a matrix $Q_k$ is computed with
complexity %of computing $Q_k$ is
comparable to that of computing
$q(Z_k)$  in \admip. Next, $L_{k+1}$ is computed as in \eqref{eq:subproblem_L}, where $U\diag(\sigma)V^T$ denotes the SVD of $Q_k$, and $\rho_k=\rho$ for all $k$.
%  is identical to
% that of \admip~ % and determined by the complexity of
% is updating a low-rank
% matrix iterate $L_{k+1}$: first, an SVD $U\diag(\sigma)V^T$ of a matrix
% $Q_k$ is computed (\asalm~iterate $Q_k$ is computed cheaply as $q(Z_k)$ is
% computed in \admip), then
% % the bottleneck
% % step is to compute
% $L_{k+1}$ is updated exactly as in
% an SVD of a low-rank matrix $Q_k=U\Diag(\sigma)V^T$ and
% then update the low-rank matrix $L_{k+1}$ as
Thus, the overall per-iteration
complexity of %the bottleneck step in
\asalm\ is comparable to that of \admip. %  for all $k\geq
% 1$.
% singular values less than the threshold $\tau=1/\rho$ is set to
% zero which results with another low-rank matrix, $L'$. %While
% $\tau=1/\rho$ for \asalm~for every
% iteration, it is $1/\rho_k$ for \admip~for
% the $k$-th iteration.
% The \asalm~code provided by the authors of~\cite{Tao09_1J} does not
% compute $L_{k+1}$ efficiently.
The \asalm~code provided by the authors of~\cite{Tao09_1J} calls the
original LANSVD
function of PROPACK which does not have the
threshold option; consequently, the \asalm\ code computes $L_{k+1}$
by first estimating its
% Indeed, first, the
rank, say $r$, and computing the leading $r$ singular values of
$Q_k$, i.e. $\sigma_1\geq\sigma_2\geq\ldots\geq\sigma_r$. %are computed using
If the $r$-th singular value $\sigma_r\leq 1/\rho$, then $L_{k+1}$ is computed
using singular-value shrinkage as in \eqref{eq:subproblem_L};
otherwise, the estimate $r$ is revised % choosing $r'>r$ (in the
% code provided
by setting $r=\min\{2r,~n\}$, and
the leading $r$ singular values of $Q_k$
are computed
\emph{from scratch}, i.e. the first $r$ that were computed previously are
simply ignored. This process is repeated until $\sigma_{r}\leq 1/\rho$. In
order to improve the efficiency of the \asalm~code and make it comparable to
\admip,
% both codes
% on the same grounds,
we used the modified LANSVD function with the
threshold option  in \emph{both} \admip~and \asalm~to compute low-rank SVDs
more efficiently. %(modified LANSVD function with threshold option
% can be downloaded from
% http://svt.stanford.edu/code.html).
This modification significantly reduced the total number singular values
computed by \asalm~when compared to the code provided by the authors
of~\cite{Tao09_1J}. %  that uses the original LANSVD function \emph{without}
% threshold option.

For all three algorithms, \admip, \admm, and \asalm, we set the
initial iterate $(Z_0,Y_0) = (\mathbf{0},\bo)$. For
\admip~the penalty multiplier sequence $\{\rho_k\}_{k\in\integers_+}$ was
chosen as follows:
\begin{equation}
\label{eq:rho_kappa}
\rho_0=\rho_1=1.25/\sigma_{\max}(\pi_{\Omega}(D)),\qquad
\rho_{k+1}=\min\{\kappa~\rho_k,\ \bar{\rho}+k\}, \quad k \geq 1  ,
\end{equation}
where $\kappa=1.25$, $\bar{\rho}=1000~\rho_0$, and $\pi_{\Omega}(\cdot)$
is defined in \eqref{eq:pi-def}. Note
that for \admm~and \asalm, $\rho_k = \rho$ for some $\rho>0$ for all $k\geq 1$.

  See Figure~\ref{alg:pseudocode} % , we display
  for an implementable
  pseudocode for \admip: line~\ref{algeq:L-problem} follows from
  \eqref{eq:subproblem_L}, and lines~\ref{algeq:S-problem}
  and~\ref{algeq:Z-problem} follow from Lemma~\ref{lem:subproblem}, since
  $\theta^*$ computed in line~\ref{algeq:theta-problem} satisfies the
  conditions given in Lemma~\ref{lem:subproblem} %when
  with
  $Q=-Y_k$,
  $\tilde{Z}=L_{k+1}$, and $\rho=\rho_k$. % Indeed, i
  % It  t is guaranteed that
  Subroutine \texttt{ThetaSearch} in Figure~\ref{alg:theta_search} uses
  the procedure outlined in the proof of Lemma~\ref{lem:subproblem}
  to compute $\theta^*$ in % according to the
  % discussion in the proof of Lemma~\ref{lem:subproblem} on how to compute
  % $\theta^*$ in
  $\cO(|\Omega|\log(|\Omega|))$ time. % It is important to
  Also, note that the roots of the quartic equation in line~\ref{algeq:quartic} of
  Figure~\ref{alg:pseudocode} can be computed in closed form using the
  formula first shown by Lodovico Ferrari, and later published in
  Cardano's Ars Magna in 1545~\cite{boyer91history}.
\subsection{Random SPCP problems}
\label{sec:random_setting}
% In this section all the numerical experiments were carried out on randomly
% generated stable principle component pursuit problems.
For a given sparsity
coefficient $c_s\in\{0.05, 0.1\}$ and a rank coefficient $c_r\in\{0.05,
0.1\}$, the data matrix  $D=L^0+S^0+N^0$ was
generated as follows:
\begin{enumerate}[i.]
\item $L^0=UV^T$, with $U\in\reals^{n\times r}$, $V\in\reals^{n\times
    r}$ for $r=\lceil c_r n\rceil$, and for all $i,j$, $U_{ij}$, $V_{ij}$, were
  independently drawn from a Gaussian distribution with mean $0$ and
  variance $1$.
\item $\Lambda\subset\{(i,j):\ 1 \leq i,j\leq n\}:=I$ was chosen uniformly at
  random such that its cardinality $|\Lambda|=\lceil c_s n^2\rceil$,
\item For each $i,j$, $S^0_{ij}$ was independently drawn from a uniform
  distribution over the interval
  $\left[-\sqrt{\frac{8r}{\pi}},\sqrt{\frac{8r}{\pi}}\right]$.
  % $S^0_{ij}\sim\mathcal{U}\left[-\sqrt{\frac{8r}{\pi}},\sqrt{\frac{8r}{\pi}}\right]$
  % for all $(i,j)\in\Lambda$ are independent random variables uniformly
  % distributed over the specified interval,
  % between $-100$ and $100$,
\item For each $i,j$, $N^0_{ij}$ was independently drawn from a Gaussian
  distribution with mean $0$ and variance $\varrho^2$.
\end{enumerate}
This construction % was motivated by the fact that we wanted
ensures that, on average, the
the magnitude of the non-zero entries of the sparse component $S^0$ % for
% $(i,j)\in\Lambda$, and
is of the same order as the entries of the
low-rank component $L^0$, i.e. $\mE[|L_{i_1 j_1}^0|] =
  \mE[|S_{i_2 j_2}^0|]$ for all $(i_1, j_1)\in I$ and for all $(i_2,
  j_2)\in\Lambda$.
% for all $(i,j)$, to have the same
% magnitude. Indeed,
% Since $U_{ij}$ and $V_{ij}$ are independently drawn from a Gaussian
% distribution with mean zero and unit variance, it follows that
% for large $n$, $L_{ij}^0\sim
% \sqrt{r}~\mathcal{N}(0,1)$. Therefore,
% $E[|L_{ij}^0|]=\sqrt{\frac{2r}{\pi}}$.
% Therefore, the way we created
% $S_{ij}^0$ for $(i,j)\in\Lambda$ ensures that
% $E[|S_{ij}^0|]=\sqrt{\frac{2r}{\pi}}$.

Let $\Omega\subset\{1,\dots,n\}\times\{1,\dots,n\}$ denote the set indices
of the observable entries of $D$, and let $\rm{SR} =\frac{|\Omega|}{n^2}$ denote
the sampling ratio of $D$. Then, the  signal-to-noise ratio % of $D$
is given by
\begin{align}
  \label{eq:snr}
  \rm{SNR}=10\log_{10}\left(\frac{E\left[\norm{\pi_{\Omega}(L^0+S^0)}_F^2\right]}
    {E\left[\norm{\pi_\Omega(N^0)}_F^2\right]}\right)=10\log_{10}\left(\frac{c_r
      n+c_s \frac{8r}{3\pi}}{\varrho^2}\right).
\end{align}
In all the numerical test problems, the value for the noise variance $\varrho^2$ % was set such that the
% signal-to-noise ratio of $D$ is at the
% given level
was set to ensure a certain
$\rm{SNR}$ level, i.e. $\varrho^2= \left(c_r n+c_s
  \frac{8r}{3\pi}\right)10^{-\rm{SNR}/10}$. We
set $\delta = \sqrt{(n + \sqrt{8n})}\varrho$~(see \cite{Tao09_1J}).

\subsubsection{\admm~vs \admip}
\begin{figure}[h!]
\centering
\includegraphics[scale=0.45]{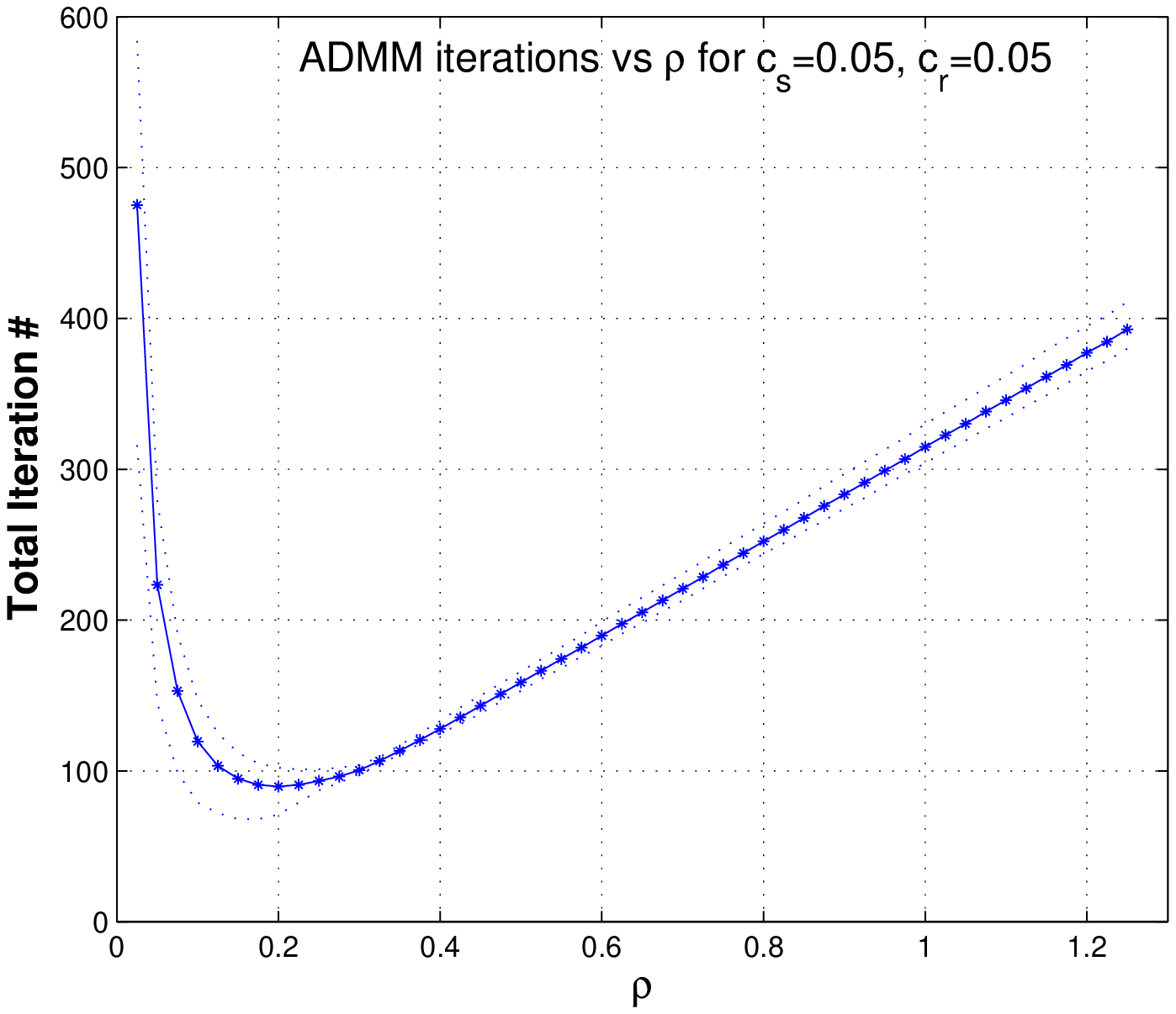}
\includegraphics[scale=0.45]{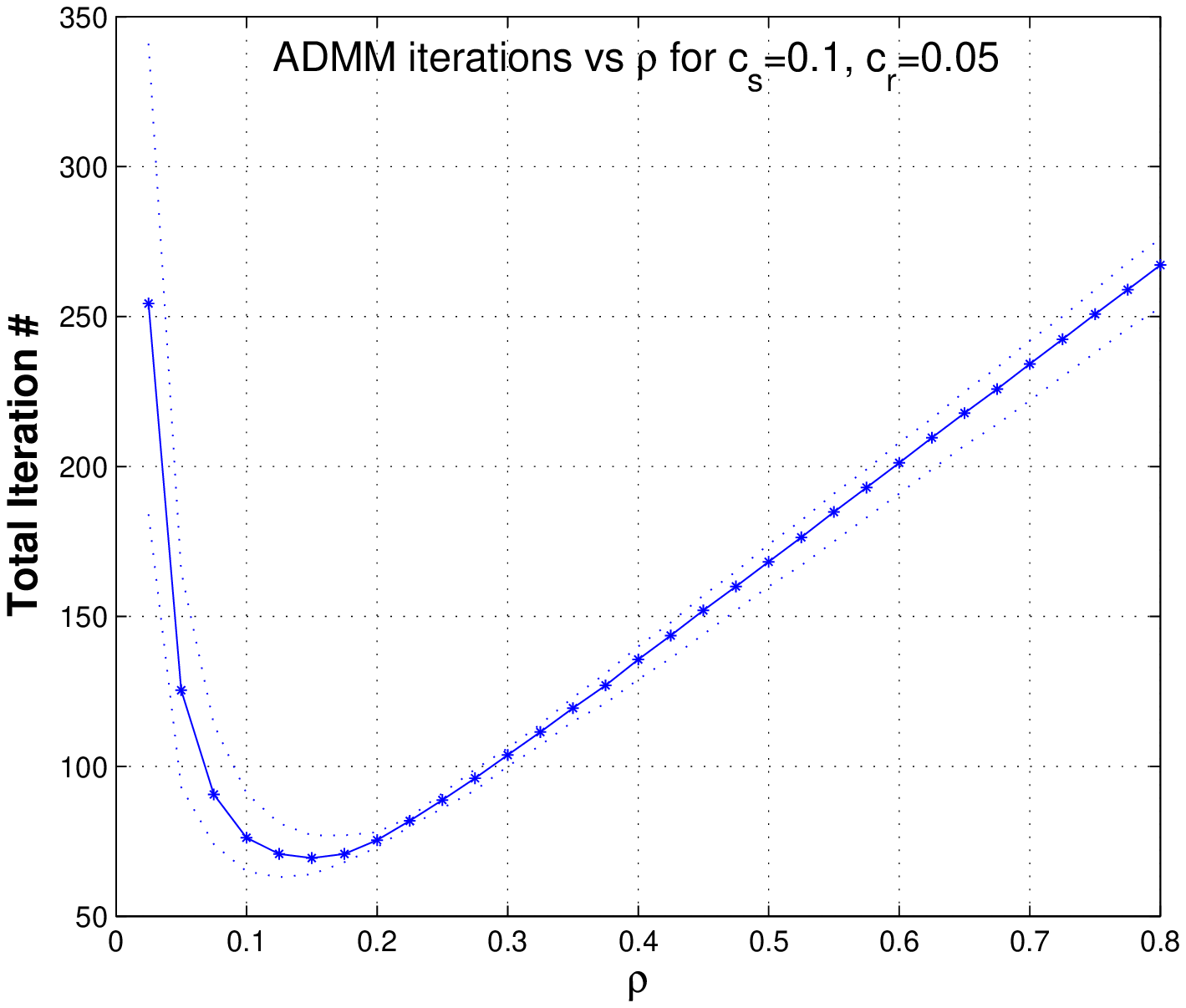}\\
\includegraphics[scale=0.45]{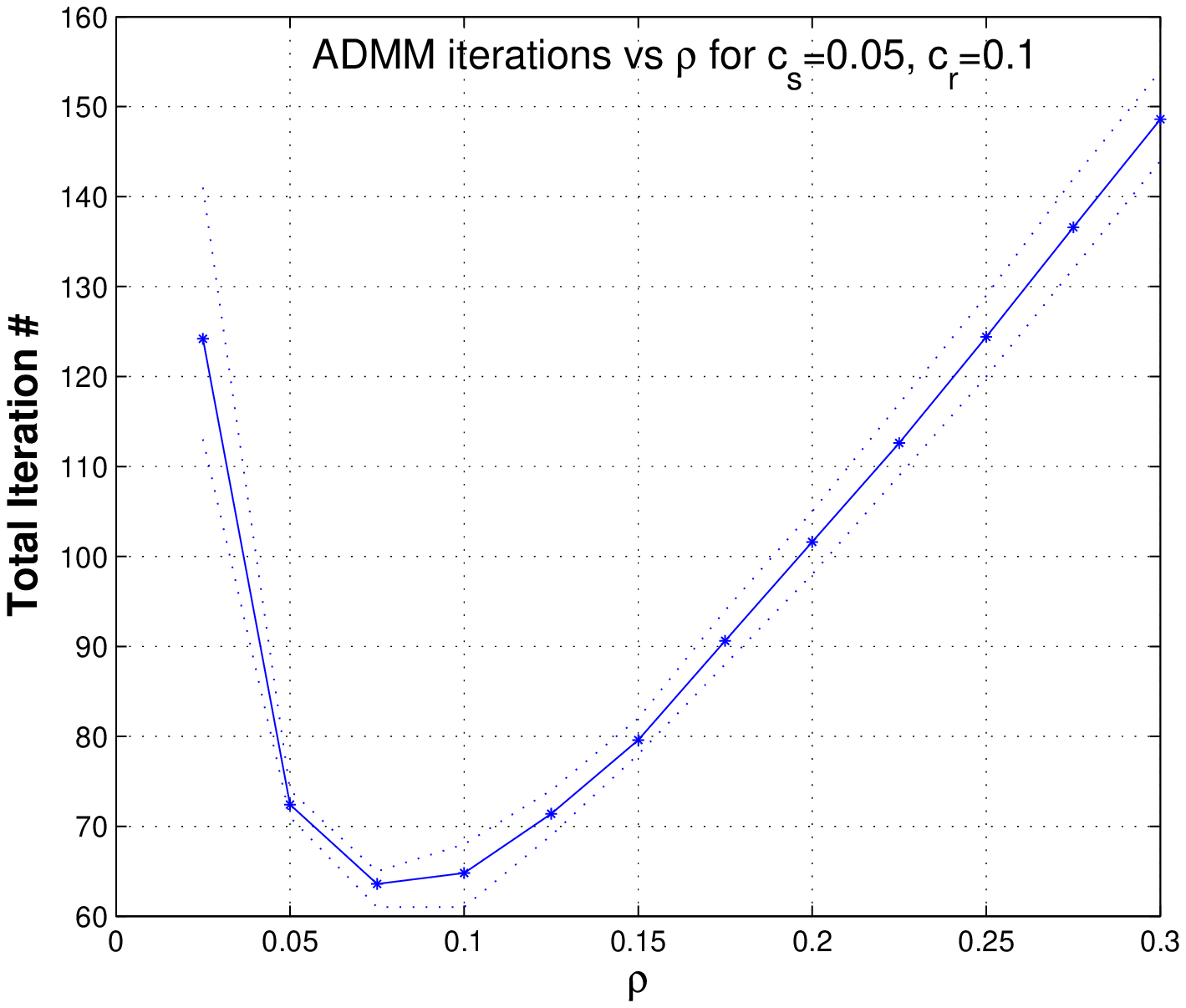}
\includegraphics[scale=0.45]{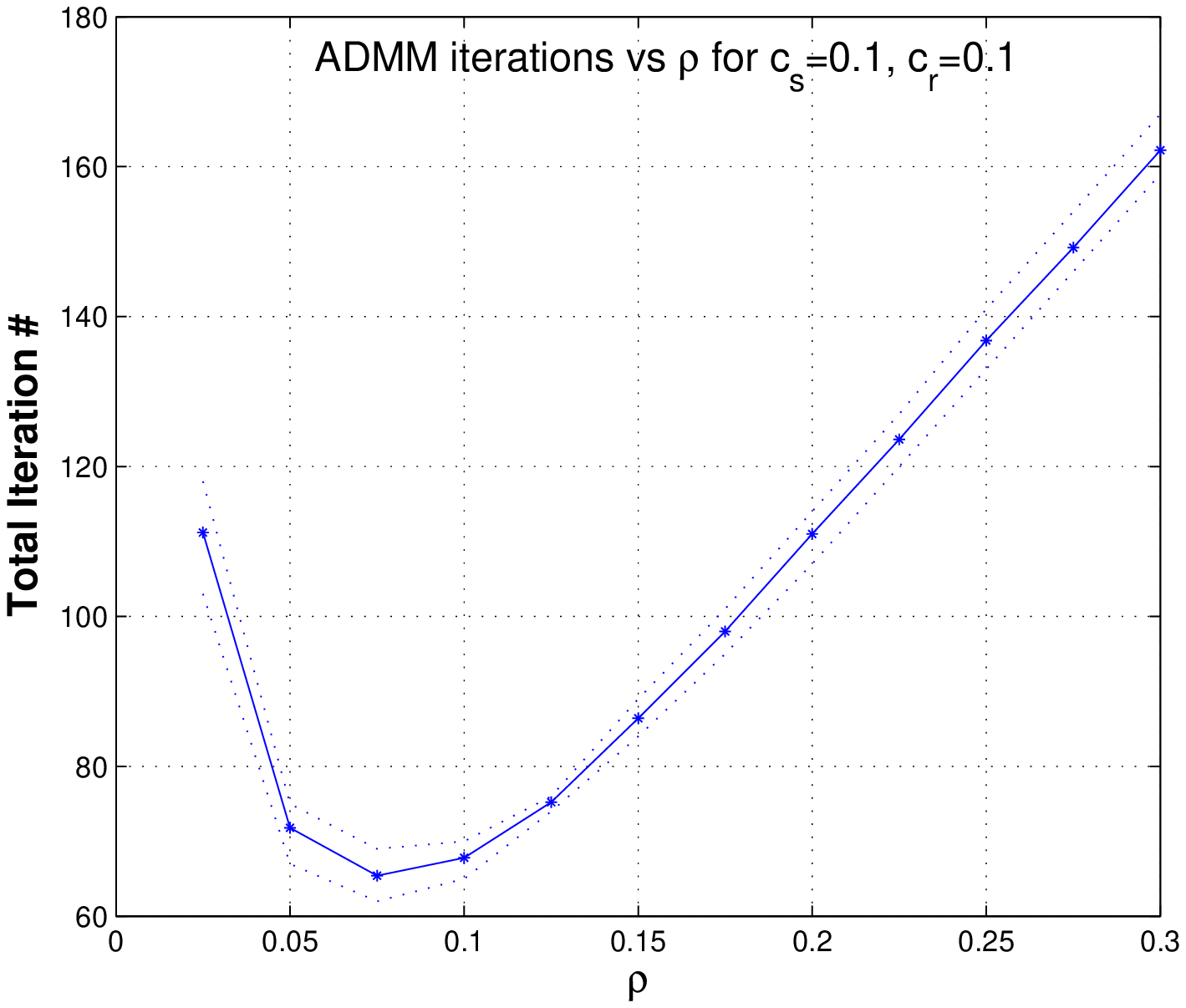}
\caption{Iteration complexity of \admm\ as a function $\rho$}
\label{fig:svd_vs_rho}
\end{figure}
We created 5 random problem instances of size $n=500$, % i.e. $D\in\reals^{n
  % \times n}$,
for each of the two choices of $c_s$ and $c_r$ such that
$\rm{SNR}=80$dB using the procedure described above in Section~\ref{sec:random_setting}. Both \admm~and
\admip~were terminated when the following primal-dual stopping condition
holds
\begin{align}
\label{eq:stopping_cond}
\frac{\norm{L_{k+1}-Z_{k+1}}_F}{\norm{D}_F}\leq \mathbf{tol}_p, \quad
      \frac{\rho_k~\norm{Z_{k+1}-Z_{k}}_F}{\norm{D}_F}\leq
            \mathbf{tol}_d.
\end{align}
See Section 3.3.1 in \cite{Boyd-etal-ADM-survey-2011} for a detailed discussion of this stopping condition.
In our experiments, we set
$\mathbf{tol}_p=\mathbf{tol}_d=8.9\times10^{-5}$ for both \admip~and
\admm. For each $c_s\in\{0.05,0.1\}$, $c_r\in\{0.05,0.1\}$, and
  penalty parameter $\rho\in\{0.025i:\ 1\leq i\leq
  50\}\subset[0.025,~1.25]$, we used \admm\ to solve $5$ random
  instances. We plot the performance of \admm\ as a function
of $\rho$ in Figure~\ref{fig:svd_vs_rho}. The solid line corresponds to
the average over the five instances, and
the dashed lines around the solid lines plot the maximum and
minimum values over the 5 random instances.
% Outside this interval the number of iterations
% required to satisfy the stopping condition \eqref{eq:stopping_cond}
% quickly blows up for our random test problems. Below we
The results of our experiments comparing \admm~with \admip~are
summarized in Table~\ref{tab:rho}. For each random problem instance,
   the reported \admm~ performance corresponds to the
  $\rho^*$ value that minimizes the number of iterations required for
  termination. The last column in Table~\ref{tab:rho} reports the range of $\rho^*$ over 5
  random instances. The column labeled \textbf{iter}
(resp. \textbf{cpu}) lists the
\emph{minimum}/\emph{\bf average}/\emph{maximum} number of total number of
iterations (resp. computation time in seconds) required to solve the $5$
instances. The columns labeled \textbf{relL} and \textbf{relS} list
the average relative error in the estimate of the low-rank component
$\norm{L^{sol}-L^0}_F/\norm{L^0}_F$ and the estimate of the sparse
component $\norm{S^{sol}-S^0}_F/\norm{S^0}_F$, respectively, where
  $(L^{sol},S^{sol})$ is the output of the particular algorithm considered. It is
clear from Table~\ref{tab:rho} that \admip~requires significantly fewer
iterations. Moreover, the range of optimal fixed penalty
  $\rho^*$ for \admm~shifts as problem parameters $c_s$ and $c_r$ change,
  making it even harder to estimate $\rho^*$. On the other hand,
  \admip~does not require tuning of any problem dependent parameter.
\vspace{-0.3cm}
%  are in the format of \emph{min}/\emph{mean}/\emph{max}
% for the \textbf{iter} and \textbf{cpu} columns, and \emph{mean} values for
% \textbf{lsv}, \textbf{relL} and \textbf{relS} columns, where \textbf{iter}
% is the total number of iteration, \textbf{lsv} is the average number of
% leading singular values computed, i.e. total number of singular values
% computed until \eqref{eq:stopping_cond} holds divided by $\mathbf{iter}$,
% \textbf{cpu} denotes the runtime of the algorithm in \emph{seconds},
% \textbf{relL} and \textbf{relS} denote the relative errors of the low-rank
% and sparse components of the solution, respectively,
% i.e.
\begin{table}[!htb]
    \begin{adjustwidth}{-2em}{-2em}
    \centering
    \caption{Comparison of \admip~and \admm}
    \renewcommand{\arraystretch}{1.1}
    {\footnotesize
    \begin{tabular}{c|c c c c c c |}
    \hline
    \textbf{Parameters}&\textbf{Algorithm}&$\mathbf{iter}$&$\mathbf{cpu}$&$\mathbf{relL}$&$\mathbf{relS}$&$\mathbf{\rho^*}$\\\hline
    \multirow{2}{*}{$\begin{array}{c}
                       \mathbf{c_s}=0.05 \\
                       \mathbf{c_r}=0.05
                     \end{array}$}
    &\textbf{ADMIP}
    & 13/\textbf{18.6}/26 & 2.1/\textbf{5.9}/11.8 & \textbf{4.7E-5} & \textbf{2.2E-4} & n/a\\ \cline{2-7}
    &ADMM
    & 68/\textbf{88.6}/101 & 16.8/\textbf{22.5}/25.1 & \textbf{3.4E-5} & \textbf{1.6E-4} & [0.15,\ 0.225]\\ \thickhline
    \multirow{2}{*}{$\begin{array}{c}
                       \mathbf{c_s}=0.1 \\
                       \mathbf{c_r}=0.05
                     \end{array}$}
    &\textbf{ADMIP}
    & 19/\textbf{20.4}/22 & 3.3/\textbf{3.6}/3.9 & \textbf{3.5E-5} & \textbf{1.3E-4} & n/a\\ \cline{2-7}
    &ADMM
    & 63/\textbf{69.2}/77 & 17.7/\textbf{20.0}/21.7 & \textbf{3.6E-5} & \textbf{1.4E-4} & [0.125,\ 0.15]\\  \thickhline
    \multirow{2}{*}{$\begin{array}{c}
                       \mathbf{c_s}=0.05 \\
                       \mathbf{c_r}=0.1
                     \end{array}$}
    &\textbf{ADMIP}
    & 14/\textbf{14}/14 & 2.2/\textbf{2.3}/2.5 & \textbf{4.9E-5} & \textbf{1.4E-4} & n/a\\ \cline{2-7}
    &ADMM
    & 61/\textbf{63}/65 & 18.3/\textbf{18.7}/19.4 & \textbf{4.8E-5} & \textbf{1.8E-4} & [0.075,\ 0.1]\\  \thickhline
    \multirow{2}{*}{$\begin{array}{c}
                       \mathbf{c_s}=0.1 \\
                       \mathbf{c_r}=0.1
                     \end{array}$}
    &\textbf{ADMIP}
    & 23/\textbf{23}/23 & 4.2/\textbf{4.2}/4.3 & \textbf{5.4E-5} & \textbf{1.6E-4} & n/a\\ \cline{2-7}
    &ADMM
    & 62/\textbf{65.4}/69 & 19.6/\textbf{21.5}/19.4 & \textbf{5.3E-5} & \textbf{1.9E-4} & [0.075,\ 0.075]\\ \hline
    \end{tabular}
    \vspace{-0.3cm}
    \label{tab:rho}
    }
    \end{adjustwidth}
\end{table}
\begin{sidewaystable}[p]
    \begin{adjustwidth}{-2em}{-2em}
    \centering
    \caption{Performance of \admip~on random test problems with missing data, SNR(D)=80dB}
    \renewcommand{\arraystretch}{1.1}
    {\footnotesize
    \begin{tabular}{c c |c c c c c |c c c c c| c c c c c|}
    \cline{3-17}
     & & \multicolumn{5}{c|}{\textbf{SR=100\%}}
     & \multicolumn{5}{c|}{\textbf{SR=90\%}}&\multicolumn{5}{c|}{\textbf{SR=80\%}}\\
    \hline
    \textbf{n}&$\mathbf{(c_s,c_r)}$
    &$\mathbf{iter}$&\textbf{lsv}&$\mathbf{cpu}$&$\mathbf{relL}$&$\mathbf{relS}$
    &$\mathbf{iter}$&\textbf{lsv}&$\mathbf{cpu}$&$\mathbf{relL}$&$\mathbf{relS}$
    &$\mathbf{iter}$&\textbf{lsv}&$\mathbf{cpu}$&$\mathbf{relL}$&$\mathbf{relS}$\\\thickhline
    \multirow{4}{*}{500}
    &(0.05,0.05) &11.6 & 35.2  &2.2 & 4.1E-5 & 1.6E-4
                 &13.2 & 35.1  &2.4 & 4.0E-5 & 1.3E-4
                 &29.0 & 78.5  &9.7 & 7.2E-5 & 4.1E-4 \\ \cline{2-17}
    &(0.1,0.05)  &17.2 & 34.8  &2.9 & 4.3E-5 & 1.8E-4
                 &17.8 & 34.8  &2.9 & 4.8E-5 & 1.7E-4
                 &19.0 & 34.7  &2.7 & 5.6E-5 & 1.6E-4 \\ \cline{2-17}
    &(0.05,0.1)  &13.0 & 58.0  &2.2 & 5.8E-5 & 1.8E-4
                 &15.6 & 58.0  &2.5 & 7.0E-5 & 1.9E-4
                 &19.8 & 58.0  &2.9 & 8.3E-5 & 2.0E-4 \\ \cline{2-17}
    & (0.1,0.1)  &21.2 & 58.0  &3.6 & 6.4E-5 & 2.2E-4
                 &23.0 & 58.0  &4.1 & 7.2E-5 & 2.2E-4
                 &25.0 & 58.0  &4.2 & 1.3E-4 & 3.6E-4 \\ \thickhline
    \multirow{4}{*}{1000}
    &(0.05,0.05) &11.0	& 61.4	& 6.7  &4.5E-5 &1.7E-4
                 &12.0  & 61.1  & 6.7  &5.4E-5 &1.6E-4
                 &14.0  & 60.6  & 6.8  &4.9E-5 &1.4E-4 \\ \cline{2-17}
    &(0.1,0.05)  &17.0	& 60.2	&11.3  &4.2E-5 &1.7E-4
                 &17.8  & 60.1  & 9.9  &4.6E-5 &1.6E-4
                 &18.8  & 60.0  & 9.3  &5.5E-5 &1.6E-4 \\ \cline{2-17}
    &(0.05,0.1)  &13.4	&105.0	& 8.5  &5.6E-5 &1.7E-4
                 &15.0  &105.0  & 7.6  &7.5E-5 &2.0E-4
                 &19.0  &105.0  & 9.3  &8.3E-5 &1.9E-4 \\ \cline{2-17}
    & (0.1,0.1)  &21.4	&105.0	&13.0  &6.3E-5 &2.2E-4
                 &23.0  &105.0  &12.0  &7.0E-5 &2.1E-4
                 &25.0  &105.0  &13.0  &8.8E-5 &2.2E-4 \\ \thickhline
    \multirow{4}{*}{1500}
    &(0.05,0.05) &11.0	&86.6	&13.2 &4.5E-5 &1.7E-4
                 &12.0  &86.2   &17.9 &5.2E-5 &1.6E-4
                 &14.0  &85.4   &17.9 &4.9E-5 &1.3E-4 \\ \cline{2-17}
    &(0.1,0.05)  &17.0	&84.6	&21.1 &4.2E-5 &1.7E-4
                 &17.6  &84.5   &26.0 &4.7E-5 &1.7E-4
                 &18.4  &84.4   &26.5 &5.9E-5 &1.7E-4 \\ \cline{2-17}
    &(0.05,0.1)  &13.4	&153.0	&22.2 &5.5E-5 &1.6E-4
                 &15.0  &153.0  &24.5 &7.2E-5 &1.9E-4
                 &19.0  &153.0  &36.3 &8.0E-5 &1.9E-4 \\ \cline{2-17}
    & (0.1,0.1)  &21.0	&153.0	&34.5 &6.3E-5 &2.2E-4
                 &23.0  &153.0  &35.6 &7.0E-5 &2.2E-4
                 &25.0  &153.0  &47.8 &8.7E-5 &2.2E-4 \\ \thickhline
    \end{tabular}
    \label{tab:missing_80dB}
    }
    \vspace{10mm}
    \centering
    \caption{Performance of \admip~on random test problems with missing data, SNR(D)=40dB}
    {\footnotesize
    \begin{tabular}{c c |c c c c c |c c c c c| c c c c c|}
    \cline{3-17}
     & & \multicolumn{5}{c|}{\textbf{SR=100\%}}
     & \multicolumn{5}{c|}{\textbf{SR=90\%}}&\multicolumn{5}{c|}{\textbf{SR=80\%}}\\
    \hline
    \textbf{n}&$\mathbf{(c_s,c_r)}$
    &$\mathbf{iter}$&\textbf{lsv}&$\mathbf{cpu}$&$\mathbf{relL}$&$\mathbf{relS}$
    &$\mathbf{iter}$&\textbf{lsv}&$\mathbf{cpu}$&$\mathbf{relL}$&$\mathbf{relS}$
    &$\mathbf{iter}$&\textbf{lsv}&$\mathbf{cpu}$&$\mathbf{relL}$&$\mathbf{relS}$\\\thickhline
    \multirow{4}{*}{500}
    &(0.05,0.05) &29.8 & 178.2  &19.2 & 6.7E-3 & 3.6E-2
                 &27.2 & 153.2  &14.6 & 6.8E-3 & 3.8E-2
                 &30.4 & 136.9  &13.8 & 7.0E-3 & 4.1E-2 \\ \cline{2-17}
    &(0.1,0.05)  &34.0 & 161.3  &19.1 & 7.5E-3 & 2.8E-2
                 &31.2 & 137.7  &14.9 & 7.6E-3 & 3.0E-2
                 &34   & 124.1  &14.8 & 7.9E-3 & 3.2E-2 \\ \cline{2-17}
    &(0.05,0.1)  &26.2 & 168.1  &14.6 & 8.1E-3 & 4.1E-2
                 &28   & 148.4  &13.4 & 8.9E-3 & 4.4E-2
                 &33   & 129.8  &13.4 & 1.0E-2 & 5.0E-2 \\ \cline{2-17}
    & (0.1,0.1)  &29.8 & 152.4  &14.9 & 9.4E-3 & 3.4E-2
                 &32   & 139.7  &15.0 & 1.0E-2 & 3.7E-2
                 &36.8 & 130.5  &15.3 & 1.2E-2 & 4.2E-2 \\ \thickhline
    \multirow{4}{*}{1000}
    &(0.05,0.05) &20.0 & 279.8 & 52.8 & 6.8E-3 & 3.6E-2
                 &21.0 & 250.5 & 48.4 & 6.8E-3 & 3.8E-2
                 &23.0 & 228.3 & 50.7 & 7.0E-3 & 4.1E-2 \\ \cline{2-17}
    &(0.1,0.05)  &25.0 & 251.8 & 62.3 & 7.6E-3 & 2.8E-2
                 &26.0 & 229.8 & 56.8 & 7.6E-3 & 3.0E-2
                 &27.0 & 200.7 & 49.9 & 7.9E-3 & 3.2E-2 \\ \cline{2-17}
    &(0.05,0.1)  &21.8 & 290.1 & 55.1 & 8.1E-3 & 4.1E-2
                 &23.0 & 255.6 & 50.5 & 8.9E-3 & 4.4E-2
                 &26.0 & 220.2 & 42.4 & 1.0E-2 & 5.0E-2 \\ \cline{2-17}
    & (0.1,0.1)  &26.8 & 269.7 & 63.0 & 9.4E-3 & 3.4E-2
                 &28.0 & 245.3 & 61.6 & 1.0E-2 & 3.6E-2
                 &29.0 & 214.1 & 48.3 & 1.2E-2 & 4.1E-2 \\ \thickhline
    \multirow{4}{*}{1500}
  &(0.05,0.05)  &20.0 &417.2 &174.0 &6.8E-3 &3.7E-2
                &21.0 &374.9 &165.0 &6.8E-3 &3.8E-2
                &21.0 &314.8 &130.4 &7.1E-3 &4.1E-2\\ \cline{2-17}
  &(0.1,0.05)   &25.0 &376.8 &198.1 &7.6E-3 &2.9E-2
                &26.0 &343.6 &189.1 &7.7E-3 &3.0E-2
                &26.0 &287.0 &148.4 &8.0E-3 &3.2E-2\\ \cline{2-17}
  &(0.05,0.1)   &22.2 &440.1 &190.0 &8.1E-3 &4.1E-2
                &23.0 &381.7 &170.2 &8.8E-3 &4.5E-2
                &26.0 &329.1 &150.6 &1.0E-2 &5.0E-2\\ \cline{2-17}
  & (0.1,0.1)   &27.0 &412.9 &211.3 &9.4E-3 &3.4E-2
                &28.0 &365.4 &204.5 &1.0E-2 &3.7E-2
                &29.0 &318.7 &164.4 &1.2E-2 &4.1E-2\\ \thickhline
    \end{tabular}
    \label{tab:missing_40dB}
    }
    \end{adjustwidth}
\end{sidewaystable}

\subsubsection{Performance of \admip\ as a function of problem parameters}
\label{sec:self_test}
Table~\ref{tab:missing_80dB} and Table~\ref{tab:missing_40dB} report the
results of the numerical experiments that we conducted to determine how
the run times and other performance measures for \admip~scale with the
problem size~$\rm{n}$, the rank of the low-rank component $\lceil{c_r
  n}\rceil$, the number of non-zero entries of the sparse component $\lceil{c_s n^2}\rceil$, the
sampling ratio $\rm{SR}$, and the $\rm{SNR}$. For this set of
experiments, we % used
% the
% stopping condition in \eqref{eq:stopping_cond} with
set the tolerances in \eqref{eq:stopping_cond} to
$\mathbf{tol_p}=\mathbf{tol_d}=1\times 10^{-4}$. %  to terminate
% \admip~iterations.

The column labeled $\mathbf{iter}$,  $\mathbf{lsv}$, $\mathbf{cpu}$,
$\mathbf{relL}$ and $\mathbf{relS}$ list, respectively, the number of iterations required
to solve the instance, the average number of leading singular values computed per iteration by
\admip, the total cpu time in second, the relative error in the low rank
component $L^{0}$, and the relative error in the low rank component
$S^{0}$, averaged over the  $5$ random
instances. Table~\ref{tab:missing_80dB} corresponds to 80dB, and
Table~\ref{tab:missing_40dB} corresponds to 40dB.
%  are reported for each choice of $\mathbf{n}$, $\mathbf{c_s}$,
% $\mathbf{c_r}$, $\mathbf{SR}$ and $\mathbf{SNR}$ values.
The results in Table~\ref{tab:missing_80dB} and
Table~\ref{tab:missing_40dB} show that the number of partial SVDs ranges from $11$ to $29$ when SNR is $80dB$, and from $20$ to $37$ when SNR is $40dB$.
%almost constant regardless of the problem dimension $\mathbf{n}$ when the
%other problem parameters are fixed.
%related to the rank and sparsity of
                                %$D$, i.e., $c_r$ and $c_p$.
Moreover, the relative error of the solution % $(L^{sol},S^{sol})$
depends only on $\rm{SNR}$ value, %  and for a fixed $\mathbf{SNR}$ value the
% relative error was
and almost independent of all the other parameters.
% almost constant for different $\mathbf{n}$,
% $\mathbf{c_s}$, $\mathbf{c_r}$ and $\mathbf{SR}$ values.

\begin{sidewaystable}[p]
    \begin{adjustwidth}{-2em}{-2em}
    \centering
    \caption{Comparison of \admip~and \asalm}
    \renewcommand{\arraystretch}{1.1}
    {\footnotesize
    \begin{tabular}{c c c |c c c c c |c c c c c |c c c c c|}
    \cline{4-18}
     & & & \multicolumn{5}{c|}{\textbf{SR=100\%}}
     & \multicolumn{5}{c|}{\textbf{SR=90\%}}&\multicolumn{5}{c|}{\textbf{SR=80\%}}\\
    \hline
    \textbf{SNR}&$\mathbf{(c_s,~c_r)}$&\textbf{Algorithm}
    &$\mathbf{iter}$&$\mathbf{lsv}$& $\mathbf{cpu}$&$\mathbf{relL}$&$\mathbf{relS}$
    &$\mathbf{iter}$&$\mathbf{lsv}$& $\mathbf{cpu}$&$\mathbf{relL}$&$\mathbf{relS}$
    &$\mathbf{iter}$&$\mathbf{lsv}$& $\mathbf{cpu}$&$\mathbf{relL}$&$\mathbf{relS}$\\\thickhline
    \multirow{8}{*}{80dB}
    &\multirow{2}{*}{$\begin{array}{c}
                       (0.05,~0.05)
                     \end{array}$}
    &\textbf{ADMIP}
    & 12	&86.2	&12.5	&3.5E-5	&1.3E-4 &13	&85.8	&12.8	&3.9E-5	&1.3E-4 &15	  &85.1	  &13.7	&4.1E-5	&1.3E-4\\ \cline{3-18}
    & &ASALM
    & 28.4	&123.9	&68.7	&4.6E-5	&4.8E-4 &29.6	&138.3	&76.9	&5.0E-5	&5.1E-4 &33.4 &146.1&50.4	&5.5E-5	&4.7E-4\\ \cline{2-18}
    &\multirow{2}{*}{$\begin{array}{c}
                       (0.1,~0.05)
                     \end{array}$}
    &\textbf{ADMIP}
    & 18	&84.4	&17.7	&3.7E-5	&1.4E-4 &18	&84.4	&17.1	&4.4E-5	&1.5E-4 &19.2 &84.2	  &16.9	&4.9E-5	&1.4E-4\\ \cline{3-18}
    & &ASALM
    & 32.4	&177.6	&109.9	&4.7E-5	&3.2E-4 &37.2	&187.1	&127.0	&4.8E-5	&2.9E-4 &42	 &194.0&83.8	&5.6E-5	&2.9E-4\\  \cline{2-18}
    & \multirow{2}{*}{$\begin{array}{c}
                       (0.05,~0.1)
                     \end{array}$}
    &\textbf{ADMIP}
    & 14.2	&153.0	&15.9	&4.9E-5	&1.4E-4 &16	&153.0	&18.6	&5.8E-5	&1.6E-4 &19	  &153.0  &20.4	&8.0E-5	&1.9E-4\\ \cline{3-18}
    & &ASALM
    & 29.2	&203.2	&86.2	&7.7E-5	&6.6E-4 &32.8	&220.0	&112.5	&8.6E-5	&6.6E-4&41	 &228.4	&79.1	&9.3E-5	&5.6E-4\\  \cline{2-18}
    & \multirow{2}{*}{$\begin{array}{c}
                       (0.1,~0.1)
                     \end{array}$}
    &\textbf{ADMIP}
    & 21	&153.0	&26.0	&6.3E-5	&2.2E-4 &23	&153.0	&26.5	&7.0E-5	&2.2E-4 &25	  &153.0  &27.1	&8.7E-5	&2.2E-4\\ \cline{3-18}
    & &ASALM
    & 34.8	&272.0	&148.4	&8.0E-5	&4.6E-4 &43	&282.5	&197.1	&8.3E-5	&3.9E-4&55	  &285.6	&138.5	&9.5E-5	&3.6E-4\\ \thickhline
    \multirow{8}{*}{40dB}
    &\multirow{2}{*}{$\begin{array}{c}
                       (0.05,~0.05)
                     \end{array}$}
    &\textbf{ADMIP}
    &7	&89.9	&10.5	&3.5E-3	&1.4E-2&8	 &88.8	&7.7	&3.7E-3	&1.5E-2&8	&88.8	&7.7	&4.3E-3	&1.6E-2\\ \cline{3-18}
    & &ASALM
    &15	&205.3	&42.1	&4.6E-3	&3.0E-2&18	&210.3	&45.1	&5.1E-03	&3.3E-02&20	  &207.1	&45.8	&5.8E-3	&3.7E-2\\ \cline{2-18}
    &\multirow{2}{*}{$\begin{array}{c}
                       (0.1,~0.05)
                     \end{array}$}
    &\textbf{ADMIP}
    &9	&87.9	&12.1	&3.8E-3	&1.5E-2&9.8 &87.3	&9.1	&4.1E-3	&1.5E-2&10	&87.2	&9.2	&4.7E-3	&1.6E-2\\ \cline{3-18}
    & &ASALM
    &20	&292.2	&78.4	&6.1E-3	&2.7E-2&24	&296.6	&81.4	&6.8E-03	&2.9E-02&28	  &285.5	&85.5	&7.4E-3	&3.1E-2\\  \cline{2-18}
    & \multirow{2}{*}{$\begin{array}{c}
                       (0.05,~0.1)
                     \end{array}$}
    &\textbf{ADMIP}
    &8	&153.0	&12.5	&5.1E-3	&1.9E-2&8.2 &153.0	&9.0	&6.0E-3	&2.1E-2&9	&153.0	&9.4	&7.6E-3	&2.5E-2\\ \cline{3-18}
    & &ASALM
    &16	&267.3	&47.1	&5.7E-3	&3.2E-2&20	&280.5	&53.5	&6.9E-03	&3.7E-02&24	  &289.7	&65.0	&8.2E-3	&4.0E-2\\  \cline{2-18}
    & \multirow{2}{*}{$\begin{array}{c}
                       (0.1,~0.1)
                     \end{array}$}
    &\textbf{ADMIP}
    &9	&153.0	&14.6	&6.1E-3	&2.0E-2&10	 &153.0	&10.9	&6.9E-3	&2.2E-2&11	&153.0	&11.9	&8.2E-3	&2.5E-2\\ \cline{3-18}
    & &ASALM
    &23	&364.6	&96.7	&7.0E-3	&2.9E-2&28	&373.5	&102.1	&7.8E-03	&3.1E-02 &35.8 &370.7	&124.1	&8.9E-3	&3.2E-2\\ \thickhline
    \end{tabular}
    \label{tab:asalm_comparison}
    }
    \end{adjustwidth}
\end{sidewaystable}
\subsubsection{\asalm~vs \admip}
We created 5 random problem instances of size $n=500$, % i.e. $D\in\reals^{n
  % \times n}$,
for each of the two choices of $c_s$, $c_r$, $\rm{SNR}$ and $\rm{SR}$ using the procedure described in Section~\ref{sec:random_setting}; and we compared \admip~with \asalm~\cite{Tao09_1J} on these random problems.
%The code for \proc{ASALM} was obtained from the authors of~\cite{Tao09_1J}.
In these numerical
tests, % , instead of using \eqref{eq:stopping_cond},
we set $\mathbf{tol}=0.05$, and terminated
\admip~using % a practical
the stopping condition
\vspace{-0.25cm}
\begin{align}
\label{eq:stopping_cond_practical}
\frac{\norm{(L_{k+1},S_{k+1})-(L_{k},S_{k})}_F}{\norm{(L_{k},S_{k})}_F+1}\leq
\mathbf{tol}~\varrho.
\end{align}
% where $\mathbf{tol}=0.05$. %and $N^0_{ij}$ are i.i.d. such that
%                            % $N^0_{ij}\sim \varrho\mathcal{N}(0,1)$ for all
%                            %$i=1,\ldots,m$ and $j=1,\ldots,n$.
% On the other hand, w
We terminated \asalm\ either when it computed a solution
with a smaller relative error compared to the \admip\ solution for the same
problem instance or when an iterate satisfied
\eqref{eq:stopping_cond_practical}. Note that this experimental setup
favors \asalm\ over \admip.
% To be more specific, let
% $D=L^0+S^0+N^0$ be generated as discussed above. For a given iterate
% $(\tilde{L},\tilde{S})$, we denote its relative error with
% $\mathbf{relL}(\tilde{L})=\norm{\tilde{L}-L^0}_F/\norm{L^0}_F$ and
% $\mathbf{relS}(\tilde{S})=\norm{\tilde{S}-S^0}_F/\norm{S^0}_F$. Suppose
% that the relative errors of the low-rank and sparse components of the
% \admip~solution are $r_L$ and $r_S$, respectively. Suppose that the
% condition \eqref{eq:stopping_cond_practical} is not satisfied within the
% first $K$ iterations, if $(K+1)$-th \asalm~iterate, $(L_{K+1},S_{K+1})$,
% is the first one to satisfy the condition
% \begin{align}
% \label{eq:stopping_cond_2}
% \mathbf{relL}(L_{K+1})\leq r_L \quad \hbox{and} \quad
% \mathbf{relS}(S_{K+1})\leq r_S,
% \end{align}
% then we let \asalm~return $(L_K,S_K)$ to be fair in terms of run time
% comparison, i.e., all the run times reported for \asalm~are the cpu times
% required for either \eqref{eq:stopping_cond_practical} or
% \eqref{eq:stopping_cond_2} to hold, whichever comes first.
% The comparison results are displayed in
% Table~\ref{tab:asalm_comparison}.
The results for the two algorithms are displayed in
Table~\ref{tab:asalm_comparison}, where the reported statistics \textbf{iter}, \textbf{cpu}, \textbf{lsv}, \textbf{relL}, and \textbf{relS} are defined in Section~\ref{sec:self_test}.
%As before,  \textbf{iter} denotes the average number of iterations
%required to solve the $5$ random instances, \textbf{cpu} denotes the
%average cpu times, \textbf{lsv} denotes the average number of leading
%singular values computed by the algorithm,  \textbf{relL} denotes the
%relative error in the low rank component $L^0$,  and \textbf{relS} denotes the
%relative error in the sparse component $S^0$.
From the results in
Table~\ref{tab:asalm_comparison},
we see that for
all of the problem classes, \asalm\ requires about \emph{twice} as many
iterations for convergence. But, the cpu time for \asalm~is considerably
larger; this difference can
be explained by the fact that on average \asalm\ computes a larger number
of leading singular values per iteration as compared to \admip.
This is clear from the
% more leading singular value
% computations than \admip~did per partial SVD, which can be seen in the
\textbf{lsv} statistics reported for both
algorithms. The results in Table~\ref{tab:asalm_comparison} also show
that although the
relative errors in the low-rank and sparse components produced by \admip~
and \asalm~were of the same order, the error of \admip~solutions
were consistently lower than those of the \asalm~solutions.
\subsection{Foreground detection problem}
\label{sec:video_test_results}
\begin{figure} [h!]
    \centering
    \mbox{\hspace{4mm}$D(t)$:}
    \includegraphics[scale=0.6]{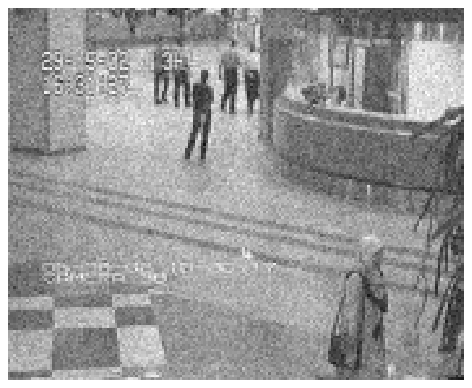}
    \includegraphics[scale=0.6]{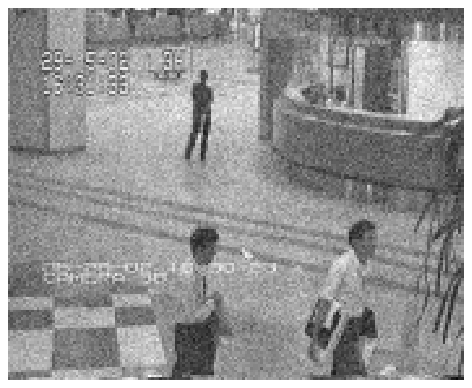}
    \includegraphics[scale=0.6]{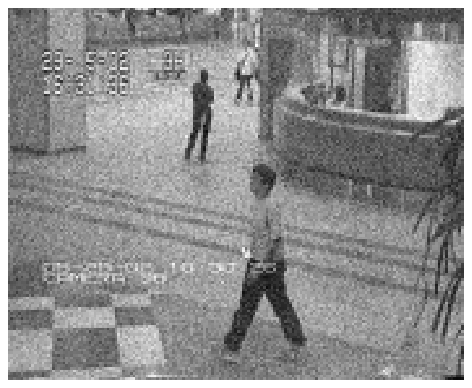}\\
    \mbox{\hspace{1mm}$L^{sol}(t)$:}
    \includegraphics[scale=0.6]{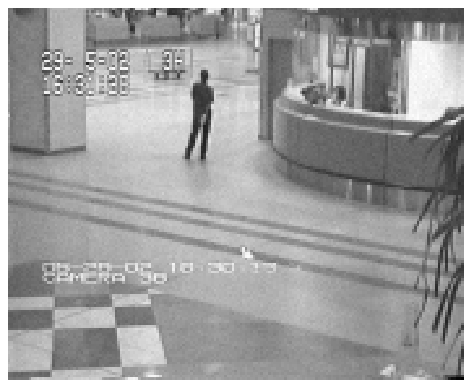}
    \includegraphics[scale=0.6]{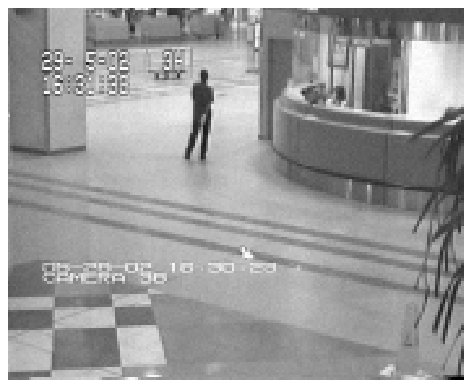}
    \includegraphics[scale=0.6]{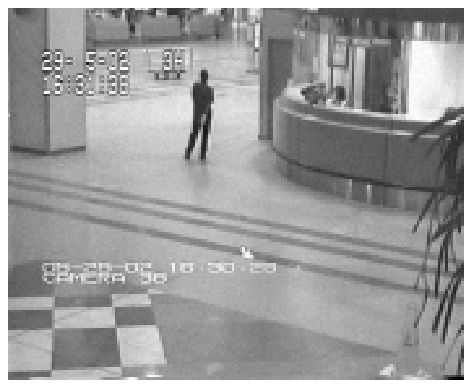}\\
    \mbox{$S^{sol}(t)$: }
    \includegraphics[scale=0.6]{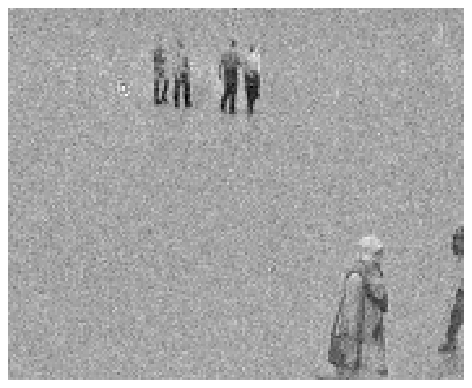}
    \includegraphics[scale=0.6]{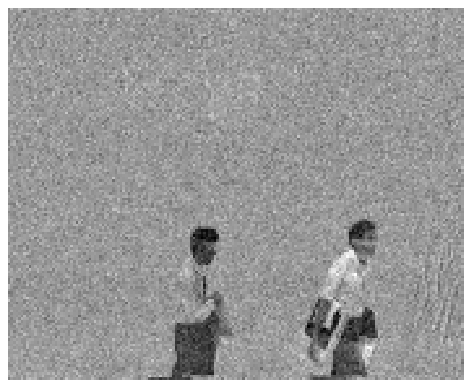}
    \includegraphics[scale=0.6]{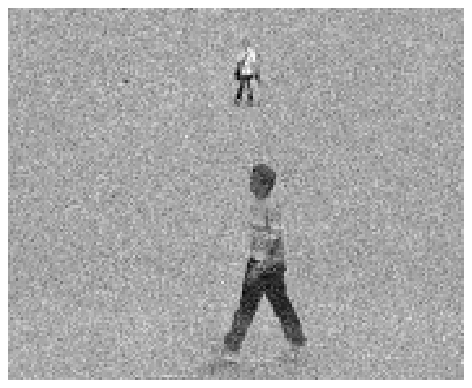}\\
    \mbox{\hspace{-1mm}$S_{post}^{sol}(t)$: }
    \includegraphics[scale=0.6]{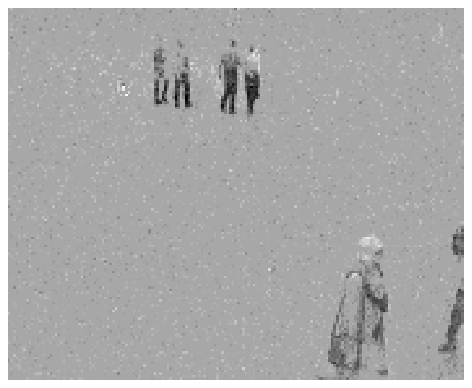}
    \includegraphics[scale=0.6]{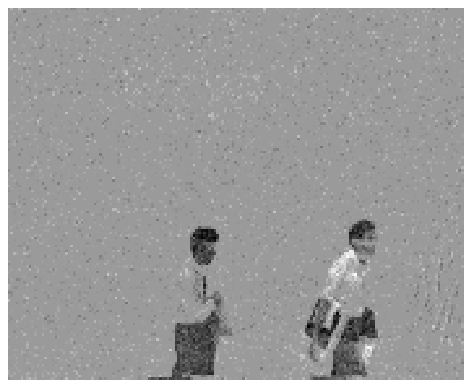}
    \includegraphics[scale=0.6]{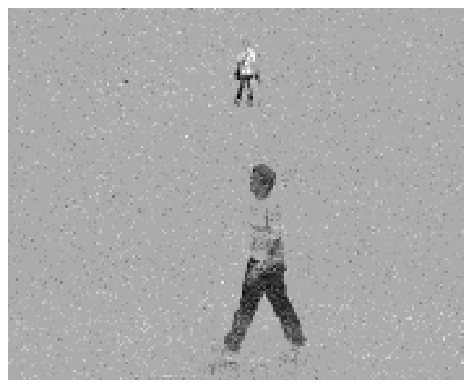}
    \caption{Background extraction from a video with $\mathbf{SNR}=20$dB and $\mathbf{SR}=100\%$ using \admip}
    \label{fig:noisy_reconstruction_test_pspg}
\end{figure}
\begin{figure} [h!]
    \centering
    \mbox{\hspace{4mm}$D(t)$:}
    \includegraphics[scale=0.6]{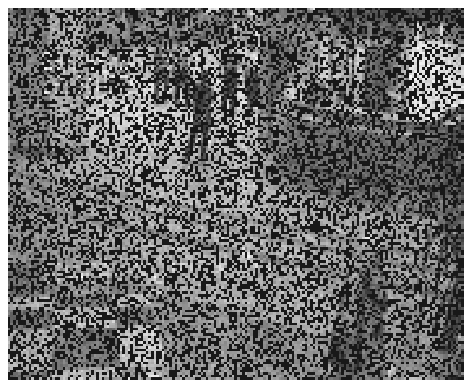}
    \includegraphics[scale=0.6]{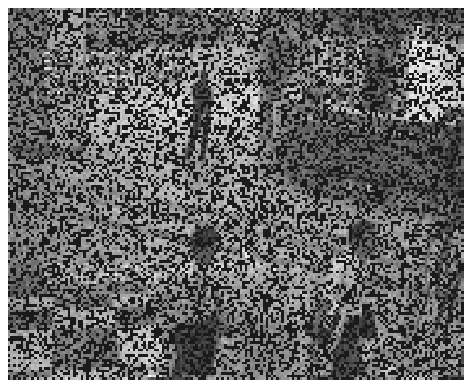}
    \includegraphics[scale=0.6]{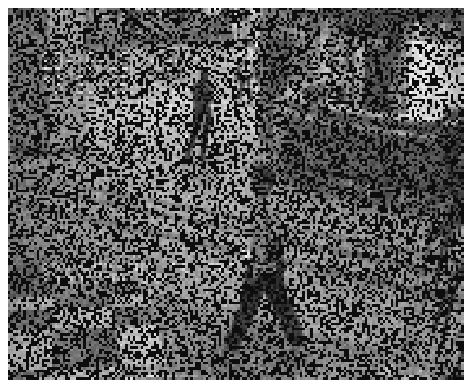}\\
    \mbox{\hspace{1mm}$L^{sol}(t)$:}
    \includegraphics[scale=0.6]{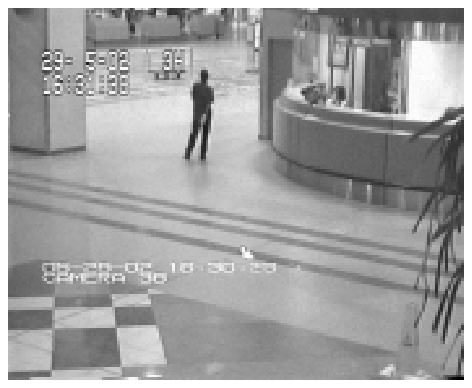}
    \includegraphics[scale=0.6]{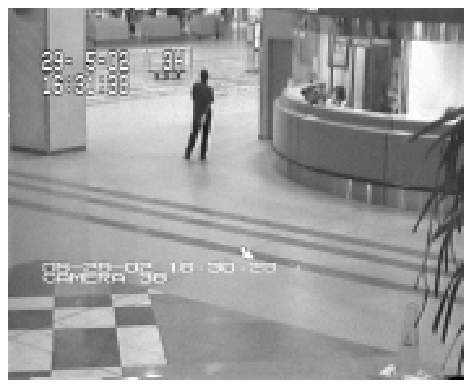}
    \includegraphics[scale=0.6]{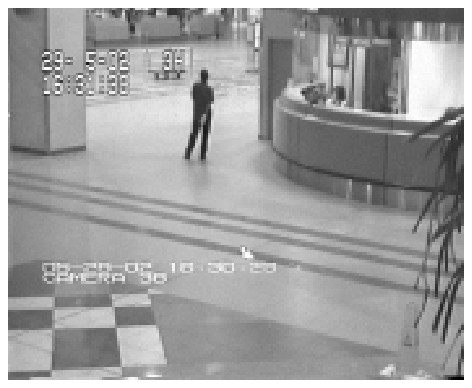}\\
    \mbox{$S^{sol}(t)$: }
    \includegraphics[scale=0.6]{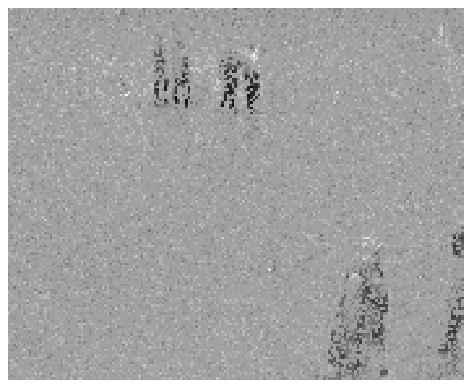}
    \includegraphics[scale=0.6]{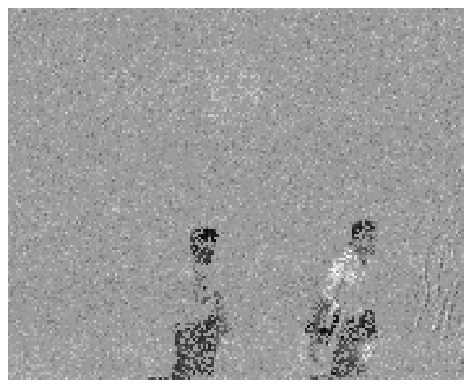}
    \includegraphics[scale=0.6]{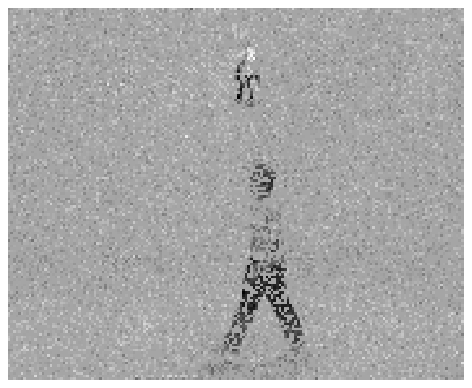}\\
    \mbox{\hspace{-1mm}$S_{post}^{sol}(t)$: }
    \includegraphics[scale=0.6]{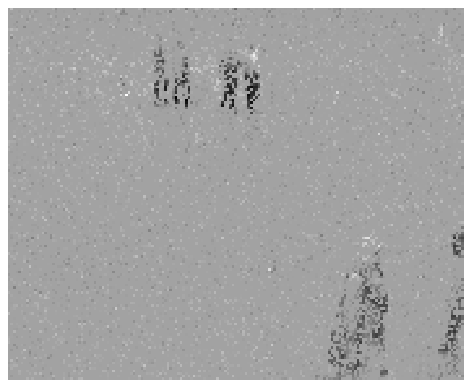}
    \includegraphics[scale=0.6]{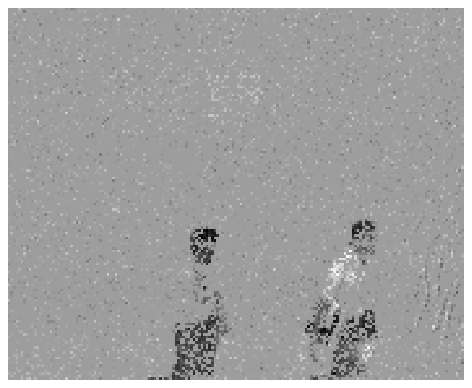}
    \includegraphics[scale=0.6]{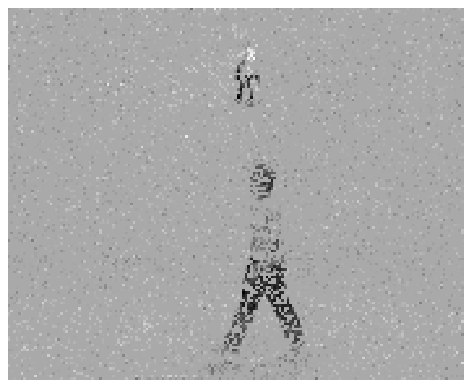}
    \caption{Background extraction from a video with $\mathbf{SNR}=20$dB and $\mathbf{SR}=60\%$ using \admip}
    \label{fig:noisy_reconstruction_test_pspg_2}
\end{figure}

Extracting the almost still background from a sequence of frames in a
noisy video is an important task in video surveillance, and it can be formulated as SPCP problem. %This problem is difficult due to the presence of moving foreground in the video. % Interestingly, this
%We show that the background can be extracted by solving an appropriately defined SPCP problem.
Let $X_t$ denote the $t$-th video frame, and $x_t\in\reals^R$ is obtained by stacking the columns of $X_t$, where $R$ is the resolution. Suppose the
background is completely stationary, and there is no measurement noise.
Then $x_t = b + f_t$, where $b$
denotes the background and $f_t$ denotes the sparse foreground in the $t$-th
frame. Let $D = [x_1,
\ldots, x_T] = b\ones^\top + [f_1, \ldots, f_T]$, i.e.
%summation of a
rank 1 matrix + sparse matrix.
% Note that by stacking the columns of each frame into a long
% vector, we can get a matrix $D$ whose columns correspond to the sequence
% of frames of the video.
In real videos, the background is never completely stationary, and there
is always measurement noise; therefore, we expect that
$D$ can be decomposed into the sum of
three matrices $D= L^0+S^0+N^0$, where $L^0$ is a low rank and $S^0$ is a sparse matrix that
represent the background and the
foreground, respectively, and $N^0$ is a dense noise matrix.% The
% matrix $L^0$, which represents the backgrounds in the frames, should be of
% low rank due to the correlation between frames. The matrix $S^0$, which
% represents the moving foregrounds in the frames, should be sparse since
% the foreground usually occupies a small portion of each frame.
\begin{table}[!ht]
    \begin{adjustwidth}{-2em}{-2em}
    \centering
    \caption{\admip~vs \asalm: Recovery statistics for foreground detection on a noisy video, $\mathbf{SNR}=20$dB}
    \renewcommand{\arraystretch}{1.75}
    {\scriptsize %\footnotesize
    \begin{tabular}{c c c c|c c c|c c c|}
    \cline{2-10}
    &\multicolumn{3}{|c|}{ASALM}&\multicolumn{3}{c|}{$\mathbf{ADMIP}~(\kappa=1.5)$}&\multicolumn{3}{c|}{$\mathbf{ADMIP} ~(\kappa=1.25)$}\\ \hline
    \multicolumn{1}{c|}{$\mathbf{SR}$}& $\mathbf{svd}$ & $\mathbf{lsv}$ & $\mathbf{cpu}$ & $\mathbf{svd}$ & $\mathbf{lsv}$ & $\mathbf{cpu}$ & $\mathbf{svd}$ & $\mathbf{lsv}$ & $\mathbf{cpu}$\\ \thickhline
    \multicolumn{1}{c|}{\textbf{100\%}} & 91 & 64.7 &198.8 & 16 & 142.5 &105.9 & 26 & 63.3 & 192.2\\ \hline
    \multicolumn{1}{c|}{\textbf{60\%}}  & 154 & 6.5 & 152.2 & 15 & 15.6 & 63.2& 24 & 14.8 & 110.3\\ \thickhline
    \end{tabular}
    \vspace{-0.2cm}
    \label{tab:compare_video}
    }
    \end{adjustwidth}
\end{table}

We used \admip~and \asalm~to extract % moving objects
the foreground in an airport surveillance
video consisting of $T=201$
grayscale $144 \times 176$  frames~\cite{Li04_1J}, i.e $R=25,344$.
 % which is a sequence of.
 % We represented the original video as a
% data matrix $D \in \reals^{(144\times 176)\times 201}$, where the $i$-th
% column of $D$ was constructed by stacking the columns of the $i$-th frame
% into a long vector.
In order to test the reconstruction performance of
both algorithms under missing data, we created a test video
by masking some of the pixels, i.e. we assumed that the sensors
corresponding to these positions were malfunctioning, and therefore, not
acquiring the signal. We also injected artificial white noise to the
remaining pixels in order to create a video with prescribed $\rm{SNR}$.
%  % that makes them shown as black,
% and injecting artificial white noise
% to the remaining pixels of the original % airport security
% video. % In
% particular, suppose that we only observe pixels corresponding to
Let $\rm{SR}$ denote the fraction of observed pixels. The locations $\Omega$ of
the observed pixels were chosen uniformly at random from the set
$\{1,\ldots,T\}\times\{1,\ldots,R\}$ such that the cardinality
$\abs{\Omega} = \lceil\rm{SR}~T~R\rceil$.
% denote the subset of
% observed pixels.
%  --this can happen if
% the sensors corresponding to pixels in $\Omega^c$ do not work, w
We created a noisy test video with $\rm{SNR}=20$dB by setting $\varrho =
\norm{\pi_\Omega(D)}_F/(\sqrt{|\Omega|}~10^{\rm{SNR}/20})$, and then for
all $(i,j)\in\Omega$ by resetting $D_{ij} = D_{ij} + N_{ij}$, % to each component $D_{ij}$ of the
% data matrix,
where each  $N_{ij}$ were independently drawn from a Gaussian
distribution with mean zero and variance $\varrho^2$.
% $\varrho\cN(0,1)$.
\admip~and \asalm~were terminated according to
\eqref{eq:stopping_cond_practical}, where $\mathbf{tol}$ is
$5\times10^{-6}$ for both \admip~and \asalm. %$1\times10^{-4}$ for \asalm, respectively, in order to obtain similar visual quality in both reconstructions.

We compared the performance of \admip~with \asalm~on the video problem
with full data $\rm{SR}=100\%$,
% i.e. $\Omega=\{1,\ldots,n_1\}\times\{1,\ldots,n_2\}$, and when
and with partial data $\rm{SR}=60\%$. On each problem instance, we ran
  \admip~ with $\kappa=1.5$ and $\kappa=1.25$, where $\kappa$ is the
  parameter that controls of the rate of growth of $\rho_k$ in
  \eqref{eq:rho_kappa}. % i.e. $|\Omega|=\lfloor0.60 n_1n_2\rfloor$.
% Let
% $(L^{sol},S^{sol})$ denote the variables corresponding to the low-rank and
% sparse components of $D$, respectively, when the
% terminates.
The frames recovered by
\asalm~were very similar to those of \admip~due to same stopping condition used; therefore, we only show the
frames recovered by \admip. The first rows in
Figure~\ref{fig:noisy_reconstruction_test_pspg} and
Figure~\ref{fig:noisy_reconstruction_test_pspg_2}
display the $35$-th, $100$-th and $125$-th frames of the noisy
surveillance video~\cite{Li04_1J} for $SR=100\%$ and $SR=60\%$,
respectively. The second and third
rows display the recovered background and foreground images of the
selected frames, respectively, using \admip. Both \admip\ and
\asalm\ were able to recover the foreground and the background fairly
accurately with only $60\%$ of the pixels functioning.
Even though the visual
quality of recovered background and foreground are very similar for
both algorithms, the statistics reported in Table~\ref{tab:compare_video}
shows that both iteration count and cpu time of \admip~are % significantly
smaller than those of \asalm. Note that, although \admip\ with $\kappa = 1.5$ has the least cpu
  time, the values for the
  $\mathbf{lsv}$
  statistic for \admip~with $\kappa=1.5$ is significantly higher % than the
                                % $\mathbf{lsv}$
  than the corresponding values for \asalm~and \admip~with
  $\kappa=1.25$. %This disparity can be explained as follows:
  Indeed, for large problem sizes, \admip\ has two different computational bottleneck. The first one is the computation of the low rank term $L_{k+1}$. For larger
  values of $\kappa$, the
  parameter $\rho_k$ grows faster; therefore, it
  follows from \eqref{eq:subproblem_L} that the number of
  leading singular values computed in each iteration grows.
  % However, large $\kappa$ results in
  % fewer iterations.
  On the other hand, in order to compute $S_{k+1}$, we need to
  % important to note
  % that at each iteration we
  sort $|\Omega|$ numbers. % with  $\cO(|\Omega|\log(|\Omega|))$
                           % complexity;
  This sorting operation with $\cO(|\Omega|\log(|\Omega|))$
  complexity becomes a computational bottleneck when $|\Omega|$ is large, especially when
  $\rm{SR}=100\%$. Moreover, large values for $\kappa$ reduces the number of iterations,
  and consequently, the number of sortings required. From the numerical
  experiments, it appears that the sorting is a computationally more critical
  step; therefore, $\kappa=1.5$ reduces the overall cpu time in comparison to $\kappa=1.25$.
% Therefore, although
  % choosing a larger $\kappa$ value will increase $\mathbf{lsv}$, the cpu
  % time goes down due to decrease in the number of iterations, leading to
  % fewer number of sorting operations.

In our preliminary numerical experiments, we noticed that the
recovered background frames are almost noise free even when the input
video was very noisy, and all the
noise shows up in the recovered foreground images. This was
observed for both \admip~and \asalm. Hence, in order to
eliminate the noise seen in the recovered foreground frames and
enhance the quality of the recovered frames, we post-process
$(L^{sol},S^{sol})$ of \admip~ as follows:
\begin{align}
\label{eq:postprocess}
S_{post}^{sol}:=\argmin_S\{\norm{S}_1:~\norm{S+L^{sol}-D}_F\leq\delta\}.
\end{align}
The fourth rows of Figure~\ref{fig:noisy_reconstruction_test_pspg} and Figure~\ref{fig:noisy_reconstruction_test_pspg_2} show the post-processed foreground frames.
\section{Conclusions}
In this paper, we propose an alternating direction method of
  multipliers with increasing penalty parameter sequence, \admip, for
  solving stable PCA problems. We prove that primal-dual iterate
sequence converges to an optimal pair when the sequence of penalty parameters
$\{\rho_k\}$ in \emph{non-decreasing}, and \emph{unbounded}. We also report numerical results comparing
\admip~with constant
penalty \admm~on synthetic random test problems and on
foreground-background separation problems. %  arising from
% background separation from surveillance video are reported.
The results clearly
show that \admip~is able to solve huge problems involving million
variables much more effectively when compared to the constant penalty
\admm. To the best of
our knowledge, % this is the first convergence result for a
\admip~is the first
variable penalty
\admm~that is guaranteed to converge to a primal-dual optimal pair when
penalties are not bounded, the objective function is
non-smooth and its subdifferential is not uniformly
bounded. However, % it should be emphasized that the
the proof of convergence of \admip~iterates heavily leverages the problem structure. In future work, % it would be
% interesting to prove convergence of variable penalty \admm~on
we plan to extend \admip~to solve a more general set of convex optimization
problems of the form $\min\{f(x)+g(y):\ Ax+By=b\}$, where $f$ and $g$ are non-smooth closed convex functions, and investigate the growth rate conditions on \emph{unbounded} $\{\rho_k\}$ that guarantee primal and dual convergence.

\section{Acknowledgements}
We would like to thank to Min Tao for providing the code \asalm.
\appendix
\section{Proofs}
\subsection{Proof of Lemma~\ref{lem:subproblem}}
\label{app:proof-1}
Suppose $\delta>0$. Let $(Z^*,S^*)$ be an optimal solution to
problem $(P_{ns})$, $\theta^*$ denote the optimal Lagrangian
multiplier for the constraint $(Z,S)\in\chi$ written as
$\frac{1}{2}\norm{\proj{Z+S-D}}^2_F\leq \frac{\delta^2}{2}$ and
$\pi^*_\Omega$ denotes the adjoint operator of $\pi_\Omega$. Note that
$\pi^*_\Omega=\pi_\Omega$. Then the KKT conditions for this
problem are given by
\begin{eqnarray}
    Q+\rho(Z^*-\tilde{Z})+\theta^*~\proj{Z^*+S^*-D}& =& 0, \label{condition1}\\
    \xi G + \theta^*~\proj{Z^*+S^*-D}& =& 0, \quad   G\in\partial\norm{S^*}_1, \label{condition2}\\
    \norm{\proj{Z^*+S^*-D}}_F& \leq &\delta, \label{condition3}\\
    \theta^* & \geq&  0, \label{condition4}\\
    \theta^*~(\norm{\proj{Z^*+S^*-D}}_F-\delta)&=&0, \label{condition5}
  \end{eqnarray}
where \eqref{condition1} and \eqref{condition2} follow from the fact
that  $\pi_\Omega \pi_\Omega=\pi_\Omega$.

From \eqref{condition1} and \eqref{condition2}, we get
\begin{align}
\label{eq:complement-components}
\projc{Z^*}=\projc{q(\tilde{Z})}, \quad \projc{G}=\mathbf{0}
\end{align}
and
\begin{eqnarray}
\left[
  \begin{array}{cc}
    (\rho+\theta^*)I &  \theta^*I\\
    \theta^*I & \theta^*I \\
  \end{array}
\right]
\left[
  \begin{array}{c}
    \proj{Z^*} \\
    \proj{S^*} \\
  \end{array}
\right]
=
\left[
  \begin{array}{c}
    \proj{\theta^*~D+\rho~q(\tilde{Z})} \\
    \proj{\theta^*~D-\xi G}\\
  \end{array}
\right], \label{eq:FTOC_1}
\end{eqnarray}
where $q(\tilde{Z})=\tilde{Z}-\rho^{-1}~Q$. From \eqref{eq:FTOC_1} it follows that
\begin{eqnarray}
\left[
  \begin{array}{cc}
    (\rho+\theta^*)I &  \theta^*I\\
    0 & \left(\frac{\rho\theta^*}{\rho+\theta^*}\right)~I \\
  \end{array}
\right]
\left[
  \begin{array}{c}
    \proj{Z^*} \\
    \proj{S^*} \\
  \end{array}
\right]
=
\left[
  \begin{array}{c}
    \proj{\theta^*~D+\rho~q(\tilde{Z})} \\
    \frac{\rho\theta^*}{\rho+\theta^*}~\proj{D-q(\tilde{Z})}-\xi \proj{G}\\
  \end{array}
\right]. \label{eq:FTOC_2}
\end{eqnarray}
From the second equation in \eqref{eq:FTOC_2}, we  get
\begin{align}
\xi\frac{(\rho+\theta^*)}{\rho\theta^*}~\proj{G}+\proj{S^*}+\proj{q(\tilde{Z})-D}=0. \label{eq:shrinkS}
\end{align}
The equation \eqref{eq:shrinkS} and $\projc{G}=\mathbf{0}$ are
precisely the first-order optimality conditions for the ``shrinkage"
problem
\eq
\min_{S\in\reals^{m\times
     n}}\left\{\xi\frac{(\rho+\theta^*)}{\rho\theta^*} \norm{S}_1+\frac{1}{2}\norm{S+\proj{q(\tilde{Z})-D}}_F^2\right\}.
\en
The expression for $S^*$ in \eqref{lemeq:S}
is the optimal solution to this ``shrinkage" problem, and  $Z^*$ given in \eqref{lemeq:Z} follows from the first equation in
\eqref{eq:complement-components} and the first row of \eqref{eq:FTOC_2}. Hence, given optimal Lagrangian dual $\theta^*$,
$S^*$ and $Z^*$ computed from equations \eqref{lemeq:S} and \eqref{lemeq:Z}, respectively, satisfy KKT conditions \eqref{condition1} and \eqref{condition2}.

Next, we show how to compute the optimal dual $\theta^\ast$. We
consider two cases.
\begin{enumerate}[(i)]
\item Suppose $\norm{\pi_\Omega\left(D-q(\tilde{Z})\right)}_F\leq\delta$. In this case, let $\theta^*=0$. Setting $\theta^*=0$ in \eqref{lemeq:S} and \eqref{lemeq:Z}, we find $S^*=\mathbf{0}$ and $Z^*=q(\tilde{Z})$. By construction, $S^*$, $Z^*$ and $\theta^*$ satisfy conditions \eqref{condition1} and \eqref{condition2}. It is easy to check that this choice of $\theta^*=0$ trivially satisfies the rest of the conditions as well. Hence, $\theta^*=0$ is an optimal lagrangian dual.
\item Next, suppose
  $\norm{\pi_\Omega\left(D-q(\tilde{Z})\right)}_F>\delta$.
From \eqref{lemeq:Z}, we have
\begin{align}
\proj{Z^*+S^*-D} = \frac{\rho}{\rho+\theta^*}~\proj{S^*+q(\tilde{Z})-D}.
\end{align}
Therefore,
\begin{align}
\norm{\proj{Z^*+S^*-D}}_F &=
\frac{\rho}{\rho+\theta^*}~\norm{\proj{S^*+q(\tilde{Z})-D}}_F,
\nonumber\\
%&=\frac{\rho}{\rho+\theta^*}~\norm{\pi_\Omega\left(\sgn\left(D-q(\tilde{Z})\right)\odot\max\left\{|D-q(\tilde{Z})|-\xi\frac{(\rho+\theta^*)}{\rho\theta^*}~E,\
%\mathbf{0}\right\}-\left(D-q(\tilde{Z})\right)\right)}_F, \nonumber\\
&=\frac{\rho}{\rho+\theta^*}
\left\|\pi_\Omega\left(\max\left\{|D-q(\tilde{Z})|
      -\xi\frac{(\rho+\theta^*)}{\rho\theta^*}
      E,\
      \mathbf{0}\right\}-|D-q(\tilde{Z})|\right)\right\|_F,\nonumber\\
&= \frac{\rho}{\rho+\theta^*}~\norm{\pi_\Omega\left(\min\left\{\xi\frac{(\rho+\theta^*)}{\rho\theta^*}~E,\ |D-q(\tilde{Z})|\right\}\right)}_F,\nonumber\\
&=\norm{\min\left\{\frac{\xi}{\theta^*}~E,\ \frac{\rho}{\rho+\theta^*}~\left|\pi_\Omega\left(D-q(\tilde{Z})\right)\right|\right\}}_F, \label{eq:Fnorm}
\end{align}
where the second equation is obtained after substituting
\eqref{lemeq:S} for $S^*$ and then componentwise dividing the
resulting expression inside the norm by
$\sgn\left(D-q(\tilde{Z})\right)$.  Define $\phi:\reals_+\rightarrow\reals$,
\begin{align}
\phi(\theta):= \norm{\min\left\{\frac{\xi}{\theta}~E,\ \frac{\rho}{\rho+\theta}~\left|\pi_\Omega\left(D-q(\tilde{Z})\right)\right|\right\}}_F.
\end{align}
It is easy to show that $\phi$ is a strictly decreasing function of
$\theta$. Since
$\phi(0)=\norm{\pi_\Omega\left(D-q(\tilde{Z})\right)}_F>\delta$ and
$\lim_{\theta\rightarrow\infty}\phi(\theta)=0$, there exists a unique
$\theta^*>0$ such that $\phi(\theta^*)=\delta$. Moreover, since $\theta^*>0$ and
$\phi(\theta^*)=\delta$, \eqref{eq:Fnorm} implies that $Z^*$, $S^*$
and $\theta^*$ satisfy the rest of KKT conditions \eqref{condition3}, \eqref{condition4} and \eqref{condition5} as well. Thus, the unique $\theta^*>0$ that satisfies $\phi(\theta^*)=\delta$ is the optimal Lagrangian dual.

We now show that $\theta^*$ can be computed in
$\cO(|\Omega|\log(|\Omega|))$ time. Let $A:=|\proj{D-q(\tilde{Z})}|$
and $0\leq a_{(1)}\leq a_{(2)}\leq ... \leq a_{(|\Omega|)}$ be the
$|\Omega|$ elements of the matrix $A$ corresponding to the indices
$(i,j)\in\Omega$ sorted in increasing order, which can be done in
$\cO(|\Omega|\log(|\Omega|))$ time. Defining $a_{(0)}:=0$ and
$a_{(|\Omega|+1)}:=\infty$, we then have for all
$j\in\{0,1,...,|\Omega|\}$ that
\begin{align}
\frac{\rho}{\rho+\theta}~a_{(j)} \leq \frac{\xi}{\theta} \leq
\frac{\rho}{\rho+\theta}~a_{(j+1)} \Leftrightarrow
\frac{1}{\xi}~a_{(j)}-\frac{1}{\rho} \leq \frac{1}{\theta} \leq
\frac{1}{\xi}~a_{(j+1)}-\frac{1}{\rho}.
\end{align}
Let $\bar{k}:=\max\left\{j: a_{(j)}\leq \frac{\xi}{\rho},\ 0\leq j\leq |\Omega| %\in\{0,1,\ldots,|\Omega|\}
   \right\}$, and for all $\bar{k}< j\leq |\Omega|$ define %$\theta_j$ such that
$\theta_j:=\frac{1}{\frac{1}{\xi}~a_{(j)}-\frac{1}{\rho}}$. Then for all $\bar{k}< j\leq
|\Omega|$, we have
\begin{align}
\phi(\theta_j)=\sqrt{\left(\frac{\rho}{\rho+\theta_j}\right)^2~\sum_{i=0}^j
  a^2_{(i)}+(|\Omega|-j)~\left(\frac{\xi}{\theta_j}\right)^2}.
\end{align}
Also define $\theta_{\bar{k}}:=\infty$ and $\theta_{|\Omega|+1}:=0$ so
that $\phi(\theta_{\bar{k}}):=0$ and
$\phi(\theta_{|\Omega|+1})=\phi(0)=\norm{A}_F>\delta$. Note that
$\{\theta_j\}_{\{\bar{k}< j\leq |\Omega|\}}$ contains all the points
at which $\phi(\theta)$ may not be differentiable for $\theta\geq 0$.
Define $j^*:=\max\{j:\ \phi(\theta_j)\leq\delta,\ \bar{k}\leq j\leq
|\Omega|\}$. Then $\theta^*$ is the unique solution of the system
\begin{align}
\label{eq:root}
\sqrt{\left(\frac{\rho}{\rho+\theta}\right)^2~\sum_{i=0}^{j^*} a^2_{(i)}+(|\Omega|-j^*)~\left(\frac{\xi}{\theta}\right)^2}=\delta \mbox{ and } \theta>0,
\end{align}
since $\phi(\theta)$ is continuous and strictly decreasing in $\theta$
for $\theta\geq 0$. Solving the equation in \eqref{eq:root} requires
finding the roots of a fourth-order polynomial (also known as a quartic
function). Lodovico Ferrari showed in 1540 that the roots of quartic functions can be
solved in closed form. Thus, it follows that  $\theta^*>0$ can be
computed in $\cO(1)$ operations.

Note that if $\bar{k}=|\Omega|$, then $\theta^*$ is the solution of the equation
\begin{align}
\sqrt{\left(\frac{\rho}{\rho+\theta^*}\right)^2~\sum_{i=1}^{|\Omega|} a^2_{(i)}}=\delta,
\end{align}
i.e. $\theta^*=
\rho\left(\frac{\norm{A}_F}{\delta}-1\right)=
\rho\left(\frac{\norm{\proj{D-q(\tilde{Z})}}_F}{\delta}-1\right)$.
\end{enumerate}
Hence, we have proved that problem~$(P_{ns})$ can be solved
efficiently when $\delta > 0$.

Now, suppose $\delta=0$. Since $\proj{Z^*+S^*-D}=0$, problem~$(P_{ns})$ %\eqref{eq:subproblem}
can be written as
\begin{equation}
\label{eq:subproblem_delta0}
\begin{array}{ll}
\min_{Z,S\in\reals^{m\times n}} & \xi\rho^{-1}
\norm{\pi_\Omega(S)}_1+\frac{1}{2}
\norm{\proj{D-S-q(\tilde{Z})}+\projc{Z-q(\tilde{Z})}}_F^2.
\end{array}
\end{equation}
Then \eqref{lemeq:LS_nonsmooth_delta0} and
$Z^*=\proj{D-S^*}+\projc{q(\tilde{Z})}$ trivially follow from first-order
optimality conditions for the above problem.
%\end{proof}
\subsection{Proof of Lemma~\ref{lem:chi_subgradient}}
\label{app:proof-2}
%\begin{replemma}{lem:chi_subgradient}
%Suppose $\delta>0$. Let $(Z^*,S^*)$ be an optimal solution to
%problem~$(P_{ns})$ in \eqref{eq:subproblem_nsa}
%and $\theta^*$ be an optimal Lagrangian multiplier such that
%$(Z^*,S^*)$ and $\theta^*$ together satisfy the KKT
%conditions. Then $(W^*,W^*)\in\partial \mathbf{1}_\chi(Z^*,S^*)$, where
%$W^*:=-Q+\rho(\tilde{Z}-Z^*)=\theta^*~\proj{Z^*+S^*-D}$.
%\end{replemma}
%\begin{proof}
Let $W^*:=-Q+\rho(\tilde{Z}-Z^*)$. Then \eqref{condition1},
\eqref{condition4} and \eqref{condition5}  % the KKT optimality conditions
%for problem \eqref{lemeq:subproblem}, which are stated
in the proof of Lemma~\ref{lem:subproblem} imply that
$W^*=\theta^*~\proj{Z^*+S^*-D}$. From the first-order optimality
conditions of $(P_{ns})$ in \eqref{eq:subproblem_nsa}, we have that
$(W^*,W)\in\partial \mathbf{1}_\chi(Z^*,S^*)$ for some
$W\in\partial\xi\norm{S^*}_1$. From \eqref{condition1} and
\eqref{condition2}, it follows that $W^*\in\partial\xi\norm{S^*}_1$. The
definition of $\chi$, chain rule on subdifferential (see Theorem 23.9 in
\cite{rockafellar1997convex}), and $W^*\in\partial\xi\norm{S^*}_1$
together imply that $(W^*,W^*)\in\partial \mathbf{1}_\chi(Z^*,S^*)$.
% and
%\begin{align}
%\norm{W^*}_F=\theta^*\norm{\proj{Z^*+S^*-D}}_F =
%\theta^*(\norm{\proj{Z^*+S^*-D}}_F-\delta)+\theta^*\delta=\theta^*\delta.
%\end{align}
%%where the last equality follows from complementary slackness,
%%i.e. $\theta^*(\norm{L^*+S^*-D}-\delta)=0$.
%Moreover, for all $(Z,S)\in\chi$, it follows from the definition of $\chi$
%that
%\eq
%\fprod{W^*,
%  \theta^*~\proj{Z+S-D}}\leq\theta^*\norm{W^*}_F
%\norm{\proj{Z+S-D}}_F\leq\theta^*\delta\norm{W^*}_F.
%\en
%Thus, for all $(Z,S)\in\chi$, we have
%$\fprod{W^*,W^*}=\norm{W^*}_F^2=\theta^*\delta\norm{W^*}_F\geq\fprod{W^*,
%  \theta^*~\proj{Z+S-D}}$.
%Hence, for all $(Z,S)\in\chi$,
%\begin{align}
%0\geq\fprod{W^*, \theta^*~\proj{Z+S-D}-W^*}=\fprod{W^*,
%  \theta^*~\proj{Z-Z^*+S-S^*}}.
%\label{eq:subgradient_key}
%\end{align}
%In the proof of Lemma~\ref{lem:subproblem} we have established that
%$\theta^*>0$
%when $\norm{\proj{D-q(\tilde{Z})}}_F>\delta$, where
%$q(\tilde{Z})=\tilde{Z}-\rho^{-1} Q$. Since
%$W^*=\theta^*~\proj{Z^*+S^*-D}$,  we have $\proj{W^*}=W^*$. Therefore, using the fact $\pi_\Omega^*=\pi_\Omega$, it follows that
%\eqref{eq:subgradient_key}
%implies that, for all $(Z,S)\in\chi$,
%\begin{align}
%0\geq\fprod{\proj{W^*}, Z-Z^*+S-S^*}=\fprod{W^*,
%  Z-Z^*+S-S^*}. \label{eq:subgradient_ineq}
%\end{align}
%On the other hand, when $\norm{\proj{D-q(\tilde{Z})}}_F\leq\delta$,
%$\theta^*=0$. Therefore,  $W^*=\theta^*~\proj{Z^*+S^*-D}=0$, and
%\eqref{eq:subgradient_ineq} follows trivially. %  Therefore,
%% \eqref{eq:subgradient_ineq} always holds and this shows that
%% $(W^*,W^*)\in\partial \mathbf{1}_\chi(Z^*,S^*)$.
%\end{proof}
\subsection{Proof of Lemma~\ref{lem:subgradients}}
\label{app:proof-3}
%\begin{replemma}{lem:subgradients}
%Let $f(\cdot):=\norm{\cdot}_*$, $g(\cdot):=\xi~\norm{\cdot}_1$ and let
%$\{L_k,Z_k,S_k,Y_k\}_{k\in\integers_+}$ denote the \admip~iterates corresponding to the sequence of penalty multipliers $\{\rho_k\}_{k\in\integers_+}$ and let $\{\hat{Y}_k\}_{k\in\integers_+}$ denote the sequence defined in \eqref{eq:yhat}. Then for
%all $k\geq 1$, $-Y_k\in\partial g(S_k)$ and
%$-\hat{Y}_k\in\partial
%f(L_k)$. Thus,
%$\{Y_k\}_{k\in\integers_+}$ and $\{\hat{Y}_k\}_{k\in\integers_+}$ are
%bounded sequences. Moreover, $\proj{Y_k}=Y_k$ for all $k\geq 1$.
%\end{replemma}
%\begin{proof}
%the $k$-th subproblem given in Step~\ref{eq:subproblem1} in
%Figure~\ref{fig:NSA} is $\min_L\norm{L}_*+\fprod{Y_k,
%L-Z_k}+\frac{\rho_k}{2}\norm{L-Z_k}_F^2$.
Since $L_{k+1}$ is the optimal solution to the subproblem in
Step~\ref{algeq:subproblem1} of \admip~corresponding to the $k$-th
iteration,  it follows that
\begin{align}
\label{eq:Lopt_cond}
0\in\partial\norm{L_{k+1}}_*+ Y_k+\rho_k(L_{k+1}-Z_k).
\end{align}
% For $k\geq 0$,
%the $k$-th subproblem given in Step~\ref{eq:subproblem2} in Figure~\ref{fig:NSA} is
%\begin{equation}
%\label{eq:subproblem_k}
%\begin{array}{ll}
%\min & \xi\norm{S}_1+\fprod{-Y_k,Z-L_{k+1}}+\frac{\rho_k}{2}\norm{Z-L_{k+1}}_F^2,\\
%s.t. & \frac{1}{2}~\norm{Z+S-D}^2_F\leq\frac{1}{2}~\delta^2:\ \ \theta_k.
%\end{array}
%\end{equation}
Let $\theta_k\geq 0$ denote the optimal Lagrange multiplier for the
quadratic
constraint in Step~\ref{algeq:subproblem2} sub-problem in the $k$-th
iteration.  Since $(Z_{k+1},S_{k+1})$ is the optimal solution, the
first-order optimality conditions imply that
\begin{align}
0\in\xi\partial\norm{S_{k+1}}_1+
\theta_k~\proj{Z_{k+1}+S_{k+1}-D}, \label{eq:Sopt_cond}\\
-Y_k+\rho_k(Z_{k+1}-L_{k+1})+\theta_k~\proj{Z_{k+1}+S_{k+1}-D}=0. \label{eq:Zopt_cond}
\end{align}
%For all $k\geq 0$, define $\hat{Y}_{k+1}:=Y_k+\rho_k(L_{k+1}-Z_k)$. Then
From \eqref{eq:Lopt_cond}, it follows that
$-\hat{Y}_{k+1}\in\partial\norm{L_{k+1}}_*$.  From \eqref{eq:Sopt_cond}
and \eqref{eq:Zopt_cond}, it follows that
$-Y_{k+1}\in\xi~\partial\norm{S_{k+1}}_1$. Since $\partial \norm{L}_*$ and
$\partial \norm{S}_1$ are uniformly bounded sets for all $L,
S\in\reals^{m\times n}$, it follows that $\{\hat{Y}_k\}_{k\in\integers_+}$
and $\{Y_k\}_{k\in\integers_+}$ are bounded sequences. Moreover,
\eqref{eq:Zopt_cond} implies that $\proj{Y_k}=Y_k$ for all $k\geq 1$.
%\end{proof}
\subsection{Proof of Lemma~\ref{lem:finite_sums}}
\label{app:proof-4}
%\begin{replemma}{lem:finite_sums}
%Suppose $\delta>0$. Let
%$\{L_k,Z_k,S_k,Y_k\}_{k\in\integers_+}$ denote the \admip~iterates
%corresponding to the non-decreasing sequence of penalty multipliers
%$\{\rho_k\}_{k\in\integers_+}$. Let
%$(L^*,L^*,S^*)=\argmin_{L,Z,S}\{\norm{L}_*+\xi~\norm{S}_1:\
%\frac{1}{2}\norm{\proj{Z+S-D}}^2_F\leq\frac{\delta^2}{2},\ L=Z\}$ denote
%any optimal solution;  $Y^*\in\reals^{m\times n}$ and $\theta^*\geq 0$
%denote any optimal Lagrangian duals corresponding to the constraints $L=Z$
%and  $\frac{1}{2}\norm{\proj{Z+S-D}}^2_F\leq\frac{\delta^2}{2}$,
%respectively. Then
%$\{\norm{Z_{k}-L^*}_F^2+\rho_{k}^{-2}\norm{Y_{k}-Y^*}_F^2\}_{k\in\integers_+}$
%is a non-increasing  sequence and
%\eq
%\begin{array}{ll}
%\sum_{k\in\integers_+}\norm{Z_{k+1}-Z_k}_F^2<\infty,
%&\sum_{k\in\integers_+}\rho_{k}^{-2}\norm{Y_{k+1}-Y_k}_F^2<\infty,\\
%\sum_{k\in\integers_+}\rho_k^{-1}\fprod{-Y_{k+1}+Y^*,
%  S_{k+1}-S^*}<\infty,
%&\sum_{k\in\integers_+}\rho_k^{-1}\fprod{-\hat{Y}_{k+1}+Y^*,
%  L_{k+1}-L^*}<\infty,\\
%\end{array}
%\en
%\eq
%\begin{array}{c}
%\sum_{k\in\integers_+}\rho_k^{-1}\fprod{Y^*-Y_{k+1},
%  L^*+S^*-Z_{k+1}-S_{k+1}}<\infty.
%\end{array}
%\en
%\end{replemma}
%\begin{proof}
For all $k \geq 0$, since $Y_{k+1}=Y_k+\rho_k(L_{k+1}-Z_{k+1})$ and
and $\hat{Y}_{k+1}:=Y_k+\rho_k(L_{k+1}-Z_k)$, we have that
$Y_{k+1}-\hat{Y}_{k+1}=\rho_k(Z_k-Z_{k+1})$. Using these relations, we obtain the following equality
\begin{eqnarray}
\lefteqn{\rho_k^{-1}\fprod{Y_{k+1}-Y_k, Y_{k+1}-Y^*}} \nonumber \\
%& = & \fprod{L_{k+1}-Z_{k+1}, Y_{k+1}-Y^*}, \nonumber \\
%& = & \fprod{L_{k+1}-L^*, Y_{k+1}-Y^*}+\fprod{L^*-Z_{k+1}, Y_{k+1}-Y^*}, \nonumber\\
%& = &\fprod{L_{k+1}-L^*, Y_{k+1}-\hat{Y}_{k+1}}+\fprod{L_{k+1}-L^*, \hat{Y}_{k+1}-Y^*}+\fprod{L^*-Z_{k+1}, Y_{k+1}-Y^*}, \nonumber \\
& = & \rho_k\fprod{L_{k+1}-L^*, Z_k-Z_{k+1}}+\fprod{L_{k+1}-L^*,
  \hat{Y}_{k+1}-Y^*}+\fprod{L^*-Z_{k+1}, Y_{k+1}-Y^*}. \label{eq:ykp-yk}
\end{eqnarray}
%Since $\norm{Z_k-L^*}_F^2=\norm{Z_{k+1}-L^*+Z_k-Z_{k+1}}_F^2$ and $\norm{Y_k-Y^*}_F^2=\norm{Y_{k+1}-Y^*+Y_k-Y_{k+1}}_F^2$, by rearranging the terms, we get
Moreover, we also have
\begin{eqnarray}
\lefteqn{\norm{Z_{k+1}-L^*}_F^2+\rho_{k}^{-2}\norm{Y_{k+1}-Y^*}_F^2} \nonumber\\
%& = & \Big(\norm{Z_{k}-L^*}_F^2 -\norm{Z_{k+1}-Z_k}_F^2+2\fprod{Z_{k+1}-L^*, Z_{k+1}-Z_k}\Big)\\
%&& \mbox{} + \rho_{k}^{-2}\Big(\norm{Y_{k}-Y^*}_F^2-\norm{Y_{k+1}-Y_k}_F^2+ 2 \fprod{Y_{k+1}-Y^*, Y_{k+1}-Y_k}\Big),\\
%\end{eqnarray*}
%Rearranging terms, we get
%\begin{eqnarray*}
%\lefteqn{\norm{Z_{k+1}-L^*}_F^2+\rho_{k}^{-2}\norm{Y_{k+1}-Y^*}_F^2}\\
& = &\norm{Z_{k}-L^*}_F^2+\rho_{k}^{-2}\norm{Y_{k}-Y^*}_F^2-\norm{Z_{k+1}-Z_k}_F^2-\rho_{k}^{-2}\norm{Y_{k+1}-Y_k}_F^2 \nonumber\\
&& \mbox{} + 2\fprod{Z_{k+1}-L^*, Z_{k+1}-Z_k} + 2 \rho_k^{-2}
\fprod{Y_{k+1}-Y_k, Y_{k+1}-Y^*}, \label{eq:lem32_quad_eq1}\\
%\end{eqnarray}
%\begin{eqnarray}
%\lefteqn{\norm{Z_{k+1}-L^*}_F^2+\rho_{k}^{-2}\norm{Y_{k+1}-Y^*}_F^2}\\
%& = & \norm{Z_{k}-L^*}_F^2+\rho_{k}^{-2}\norm{Y_{k}-Y^*}_F^2
%-\norm{Z_{k+1}-Z_k}_F^2-\rho_{k}^{-2}\norm{Y_{k+1}-Y_k}_F^2, \nonumber\\
%&& \mbox{} + 2\fprod{Z_{k+1}-L^*, Z_{k+1}-Z_k} + 2\fprod{L_{k+1}-L^\ast,Z_k-Z_{k+1}} \nonumber\\
%&&-2\rho_k^{-1}\left(\fprod{-\hat{Y}_{k+1}+Y^*,L_{k+1}-L^*}+\fprod{-Y_{k+1}+Y^*, L^*-Z_{k+1}}\right), \nonumber\\
%\end{eqnarray}
%Since $\fprod{Z_{k+1}-L^*, Z_{k+1}-Z_k} + \fprod{L_{k+1}-L^\ast, Z_k-Z_{k+1}}=\fprod{Z_{k+1}-L_{k+1}, Z_{k+1}-Z_k}$, we get
%\begin{eqnarray*}
%\lefteqn{\norm{Z_{k+1}-L^*}_F^2+\rho_{k}^{-2}\norm{Y_{k+1}-Y^*}_F^2}\\
& = &\norm{Z_{k}-L^*}_F^2 +
\rho_{k}^{-2}\norm{Y_{k}-Y^*}_F^2-\norm{Z_{k+1}-Z_k}_F^2
-\rho_{k}^{-2}\norm{Y_{k+1}-Y_k}_F^2,\nonumber\\
&& \mbox{} +2\fprod{Z_{k+1}-L_{k+1},
  Z_{k+1}-Z_k}-2\rho_k^{-1}\left(\fprod{-\hat{Y}_{k+1}+Y^*,
    L_{k+1}-L^*}+\fprod{-Y_{k+1}+Y^*, L^*-Z_{k+1}}\right), \nonumber\\
%\end{eqnarray}
%\begin{eqnarray}
%\lefteqn{\norm{Z_{k+1}-L^*}_F^2+\rho_{k}^{-2}\norm{Y_{k+1}-Y^*}_F^2} \nonumber\\
& =&\norm{Z_{k}-L^*}_F^2 +\rho_{k}^{-2}\norm{Y_{k}-Y^*}_F^2
-\norm{Z_{k+1}-Z_k}_F^2-\rho_{k}^{-2}\norm{Y_{k+1}-Y_k}_F^2, \nonumber\\
&&-2\rho_k^{-1}\left(\fprod{Y_{k+1}-Y_k, Z_{k+1}-Z_k} +
  \fprod{-\hat{Y}_{k+1}+Y^*, L_{k+1}-L^*}+\fprod{-Y_{k+1}+Y^*,
    L^*-Z_{k+1}}\right), \label{eq:preinduction_step}
\end{eqnarray}
where the second equality follows from rewriting the last term in \eqref{eq:lem32_quad_eq1} using \eqref{eq:ykp-yk}, and the last equality follows from the relation $L_{k+1}-Z_{k+1} = \rho_k^{-1}(Y_{k+1}-Y_k)$.

Since $Y^*$ and $\theta^*$ are optimal Lagrangian dual variables, we have
\begin{align*}
(L^*,L^*,S^*)=\argmin_{L,Z,S}\norm{L}_*+\xi~\norm{S}_1 +\fprod{Y^*,
  L-Z}+\frac{\theta^*}{2}\left( \norm{\proj{Z+S-D}}^2_F-\delta^2\right).
\end{align*}
From first-order optimality conditions, we get
\begin{eqnarray*}
0& \in & \partial\norm{L^*}_*+Y^*,\\
0& \in & \xi~\partial\norm{S^*}_1+\theta^*~\proj{L^*+S^*-D},\\
0 & = & -Y^*+\theta^*~\proj{L^*+S^*-D}.
\end{eqnarray*}
Hence, $-Y^*\in\partial\norm{L^*}_*$ and
$-Y^*\in\xi~\partial\norm{S^*}_1$. Moreover, from
Lemma~\ref{lem:subgradients}, we also have that $-Y_k\in\partial
\xi~\norm{S_k}_1$ for all $k\geq 1$. Since $\xi~\norm{.}_1$ is convex, it
follows that
\begin{align}
\fprod{-Y_{k+1}+Y_k, S_{k+1}-S_k}\geq 0, \label{eq:S_subgradient_monotone}\\
\fprod{-Y_{k+1}+Y^*, S_{k+1}-S^*}\geq 0. \label{eq:Sopt_subgradient_monotone}
\end{align}
Since $\rho_{k+1}\geq\rho_k$ for all $k\geq 1$, first adding
\eqref{eq:S_subgradient_monotone} to \eqref{eq:preinduction_step}, then
adding and subtracting \eqref{eq:Sopt_subgradient_monotone}, we get
\begin{eqnarray}
\lefteqn{\norm{Z_{k+1}-L^*}_F^2 +\rho_{k+1}^{-2}\norm{Y_{k+1}-Y^*}_F^2}
\nonumber\\
%& \leq &\norm{Z_{k+1}-L^*}_F^2 +\rho_{k}^{-2}\norm{Y_{k+1}-Y^*}_F^2,\nonumber\\
& \leq &\norm{Z_{k}-L^*}_F^2 +\rho_{k}^{-2}\norm{Y_{k}-Y^*}_F^2
-\norm{Z_{k+1}-Z_k}_F^2-\rho_{k}^{-2}\norm{Y_{k+1}-Y_k}_F^2
\nonumber\\
&& \mbox{} -2\rho_k^{-1}\left(\fprod{-\hat{Y}_{k+1}+Y^*,
    L_{k+1}-L^*}+\fprod{-Y_{k+1}+Y^*, S_{k+1}-S^*}\right) \nonumber\\
&& \mbox{} -2\rho_k^{-1}\left(\fprod{Y_{k+1}-Y_k,
    Z_{k+1}+S_{k+1}-Z_k-S_k}+\fprod{-Y_{k+1}+Y^*,
    L^*+S^*-Z_{k+1}-S_{k+1}}\right). \label{eq:induction_step}
\end{eqnarray}
Lemma~\ref{lem:chi_subgradient} applied to the
Step~\ref{algeq:subproblem2} sub-problem corresponding to the $k$-th
iteration gives $(Y_{k+1},Y_{k+1})\in\partial \mathbf{1}_{\chi}(Z_{k+1},S_{k+1})$.
Using an argument similar to  that  used in the proof of
Lemma~\ref{lem:chi_subgradient}, one can also show that $(Y^*,Y^*)\in\partial \mathbf{1}_{\chi}(L^*,S^*)$.
Moreover, since $-Y^*\in \partial
\xi~\norm{S^*}_1$, $-Y^*\in\partial\norm{L^*}_*$, and $-Y_{k}\in\partial \xi~\norm{S_k}_1$,
$-\hat{Y}_{k}\in\partial\norm{L_k}_*$ for all $k\geq 1$, we have that for all
$k \geq 0$,
\begin{align*}
\fprod{Y_{k+1}-Y_k, Z_{k+1}+S_{k+1}-Z_k-S_k}\geq 0,\\
\fprod{-Y_{k+1}+Y^*, L^*+S^*-Z_{k+1}-S_{k+1}}\geq 0,\\
\fprod{-Y_{k+1}+Y^*, S_{k+1}-S^*}\geq 0,\\
\fprod{-\hat{Y}_{k+1}+Y^*, L_{k+1}-L^*}\geq 0.
\end{align*}
This set of inequalities and \eqref{eq:induction_step} together imply that
$\{\norm{Z_{k}-L^*}_F^2+\rho_{k}^{-2}\norm{Y_{k}-Y^*}_F^2\}_{k\in\integers_+}$
is a non-increasing sequence. Using this fact,  rewriting
\eqref{eq:induction_step} and summing over $k\in\integers_+$, we get
\begin{align*}
&
\sum_{k\in\integers_+}\norm{Z_{k+1}-Z_k}_F^2+\rho_{k}^{-2}\norm{Y_{k+1}-Y_k}_F^2
\\
& \mbox{} +
2\sum_{k\in\integers_+}\rho_k^{-1}\left(\fprod{-\hat{Y}_{k+1}+Y^*,
    L_{k+1}-L^*}+\fprod{-Y_{k+1}+Y^*, S_{k+1}-S^*}\right)\\
&\mbox{} + 2\sum_{k\in\integers_+}\rho_k^{-1}\left(\fprod{Y_{k+1}-Y_k,
    Z_{k+1}+S_{k+1}-Z_k-S_k}+\fprod{-Y_{k+1}+Y^*,
    L^*+S^*-Z_{k+1}-S_{k+1}}\right)\\
\leq &
\sum_{k\in\integers_+}\left(\norm{Z_{k}-L^*}_F^2
  +\rho_{k}^{-2}\norm{Y_{k}-Y^*}_F^2-\norm{Z_{k+1}-L^*}_F^2
  -\rho_{k+1}^{-2}\norm{Y_{k+1}-Y^*}_F^2\right)<\infty.
\end{align*}
This inequality is sufficient to prove the rest of the lemma.
%\end{proof}
\bibliographystyle{siam}
\bibliography{All}
\end{document}